\documentclass[a4paper,12pt]{amsart}
\usepackage[hiresbb]{graphicx}
\usepackage{amsmath,amssymb,amsfonts,amsthm,color}
\newtheorem{definition}{Definition}[section]
\newtheorem{theorem}[definition]{Theorem}

\newtheorem{corollary}[definition]{Corollary}
\newtheorem{remark}[definition]{Remark}
\newtheorem{example}[definition]{Example}
\newtheorem{proposition}[definition]{Proposition}

\newcommand{\M}{\mathcal{M}}
\newcommand{\R}{\mathbb{R}}
\newcommand{\C}{\mathbb{C}}
\newcommand{\N}{\mathbb{N}}

\begin{document}

\title[Law of large numbers for non-linear traces]{Law of large numbers for non-linear traces of the 
Choquet type on finite factors}
%\author{Masaru Nagisa and Yasuo Watatani}

\author{Masaru Nagisa}
\address[Masaru Nagisa]{Department of Mathematics and Informatics, Faculty of Science, Chiba University, 
Chiba, 263-8522,  Japan: \ Department of Mathematical Sciences, Ritsumeikan University, Kusatsu, Shiga, 525-8577,  Japan}
\email{nagisa@math.s.chiba-u.ac.jp}
\author{Yasuo Watatani}
\address[Yasuo Watatani]{Department of Mathematical Sciences,
Kyushu University, Motooka, Fukuoka, 819-0395, Japan}
\email{watatani@math.kyushu-u.ac.jp}

\maketitle

\begin{abstract}
We introduced non-linear traces of the Choquet type and the Sugeno type on semi-finite factors $\M$ in \cite{nagisawatatani4}
as a non-commutative analog of the Choquet integral and Sugeno integral for non-additive measures. 
We need a weighted dimension function on the projections of $\M$, which is an analog of 
a monotone measure.  
In this paper, we study the law of large numbers for non-linear traces of the Choquet type on finite factors $\M$.   
Since averages do not converge in general, we study the range of their 
accumulation points, that is, we estimate their limit supremum and limit infimum.  We examine the trials of sequences 
consisting of 
self-adjoint operators, which appear in coin toss or Powers' binary shifts. We have also found some unexpected examples of Powers' binary shifts which satisfy what we call 
 {\it the uniform norm law of large numbers}. 
This is an attempt at non-linear and non-commutative probability theory on matrix algebras and factors of 
type ${\rm II}_1$. 


\medskip\par\noindent
Key words: law of large numbers, non-linear trace, finite factor, was
non-commutative probability theory

\end{abstract}

\section{Introduction}

We studied several classes of general non-linear positive maps between $C^*$-algebras 
in \cite{nagisawatatani}.  We introduced non-linear traces  
of the Choquet type and the Sugeno type on matrix algebras in \cite{nagisawatatani2}, on the algebra $K(H)$ of 
compact operators in \cite{nagisawatatani3}, and on semi-finite factors in \cite{nagisawatatani4}.   

Ando-Choi \cite{A-C} and Arveson \cite{Ar2} initiated the study of  non-linear completely 
positive maps.   
Hiai-Nakamura \cite{H-N} studied a non-linear counterpart of Arveson's 
Hahn-Banach type extension theorem \cite {Ar1}. 
Bel\c{t}it\u{a}-Neeb \cite{B-N} studied non-linear completely positive maps and dilation theorems for real involutive algebras.
Dadkhah-Moslehian \cite{D-M} studied some properties of non-linear positive maps 
like Lieb maps and the multiplicative domain for 3-positive maps. 
Dadkhah-Moslehian-Kian \cite{D-M-K} examined the continuity of non-linear positive maps between $C^*$-algebras.

The functional calculus by a continuous positive function is an important example of non-linear positive maps as in  
\cite{bhatia1}, \cite{bhatia2}, and  \cite{H}. %, \cite{D} and \cite{Si}. 
In particular, the functional calculus by operator monotone functions is crucial for studying operator means 
in Kubo-Ando theory \cite{kuboando}.  
Another important motivation for the study of non-linear positive maps on $C^*$-algebras is
non-additive measure theory, which was initiated by Sugeno \cite{Su} and  Dobrakov \cite{dobrakov}. 
Choquet integrals \cite{Ch} and Sugeno integrals \cite{Su} are studied as non-linear integrals in 
non-additive measure theory, see \cite{wang} for example.  
They have partial additivities. 
Choquet integrals have  comonotonic additivity and 
Sugeno integrals have comonotonic F-additivity (fuzzy additivity). Conversely,  it is known that 
these partial additivities characterize 
Choquet integrals and Sugeno integrals, see, for example, \cite{De}, 
\cite{schmeidler}, \cite{C-B}, \cite{C-M-R}, \cite{D-G}. 
We also note that the inclusion-exclusion integral by Honda-Okazaki \cite{H-O} for non-additive monotone measures is a general notion of non-linear integrals. 

In our paper \cite{nagisawatatani4}, we introduced non-linear traces of the Choquet type and Sugeno type on a semifinite factor $M$ as a non-commutative analog of the Choquet integral and Sugeno integral for non-additive measures. 
We need a weighted dimension function $p \mapsto \alpha(\tau(p))$ for any projection $p \in M$, which is an analog of 
a monotone measure.

On the other hand, both non-additive probabilities and the law of large numbers are studied in Bernoulli's 
{\it Ars Conjectandi} (1713).  Recently, in various fields,  several people have studied the law of large numbers for 
non-additive probability measures, for example, see \cite{chen}, \cite{M-M}, \cite{peng}, \cite{teran}, and 
their references. 

In this paper, we study the law of large numbers for non-linear traces of the Choquet type on finite factors $\M$ 
under the assumption that $\alpha$ is concave or convex. Since averages do not converge in general by non-linearity, 
we study the range of their accumulation points, that is, we estimate their limit supremum and limit infimum. We study the cases of coin tosses 
which appear in a maximal abelian subalgebra of the AFD factor of type $II_1$  and sequences of self-adjoint unitaries 
of Powers' binary shifts.   We have also found some unexpected examples of Powers' binary shifts which satisfy what we call 
the  {\it the uniform norm law of large numbers}.  To study the law of large numbers for sequences of self-adjoint operators, 
we need to consider extensions of non-linear traces of the Choquet type on the positive operators to the self-adjoint operators.  
The extensions are non-trivial due to the non-linearity of traces of the Choquet type.  We remark that we have Chebyshev's inequality for non-linear traces of the Choquet type.  

This is an attempt at non-linear and non-commutative probability theory on matrix algebras and factors of 
type $II_1$.

\vspace{3mm}

This work was supported by JSPS KAKENHI Grant Numbers \linebreak
JP17K18739 and JP23K03151
and also supported by the Research Institute for Mathematical Sciences, an International Joint Usage/Research Center located in Kyoto University.

%%%
%%%      Section 2
%%%

\section{Non-linear traces of the Choquet type}
In our paper \cite{nagisawatatani4}, we introduced non-linear traces of the Choquet type and Sugeno type on a semifinite factor $M$ as a non-commutative analog of the Choquet integral and Sugeno integral for non-additive measures. 
We need a weighted dimension function $p \mapsto \alpha(\tau(p))$ for any projection $p \in M$, which is an analog of 
a monotone measure.  In this paper,  we only consider non-linear traces of the Choquet type on a {\it finite factor} $\M$.  
We begin to recall the notion of non-linear traces of the Choquet type.  
%%%%%%%%%%%%%%%%%%%%%%%%%%%%%
%     Definition 2.1
%%%%%%%%%%%%%%%%%%%%%%%%%%%%%
\begin{definition} \rm
Let $\Omega$ be a set and ${\mathcal B}$ a ring of sets on $\Omega$, that is, 
${\mathcal B}$ is a family of subsets of $\Omega$ which is closed under the operations 
union $\cup$ and set theoretical difference $\backslash$. Hence  ${\mathcal B}$ is also 
closed under intersection $\cap$. 
A function $\mu: {\mathcal B} \rightarrow [0, \infty]$ is  called a 
{\it monotone measure} if $\mu$ satisfies 
\begin{enumerate}
\item[$(1)$]  $\mu(\emptyset) = 0$, and 
\item[$(2)$] For any $A,B \in {\mathcal B}$, if $A \subset B$, 
then $\mu(A) \leq \mu(B)$. 
\end{enumerate}
If $\Omega$ is not a discrete space, we often assume that ${\mathcal B}$ is a $\sigma$-field.

\end{definition}

\begin{definition} \rm (Choquet integral)
Let $\Omega$ be a set and ${\mathcal B}$ a $\sigma$-field on $\Omega$.  
Let  $\mu: {\mathcal B} \rightarrow [0, \infty]$ be a monotone measure on $\Omega$. 
Let $f$ be a non-negative measurable function on $\Omega$ in the sense that $f^{-1}(A) \in {\mathcal B}$ 
for any Borel set $A$ on $[0, \infty]$. Then the Choquet integral of $f$ is defined by 
\begin{align*}
{\rm (C)} \int f d\mu & := \int_0^{\infty} \mu ( \{ x \in \Omega \ | \ f(x) \geq s \}) ds \\
& =  \int_0^{\infty} \mu ( \{ x \in \Omega \ | \ f(x) > s \}) ds.  
\end{align*}
\end{definition}

\begin{remark} \rm  If $f$ is a simple function with 
$$
f = \sum_{i=1}^{n} a_i \chi _{A_i},  \ \ \  A_i \cap A_j = \emptyset  \  ( i\not=j),  
$$
then the Choquet integral of $f$ is given by 

$$
{\rm (C)}\int f d\mu = \sum_{i=1}^{n-1} (a_{\sigma(i) }- a_{\sigma(i+1)})\mu(B_i) 
+ a_{\sigma(n) }\mu(B_n) , 
$$
where $\sigma$ is a permutation on $\{1,2,\dots n\}$ such that 
$a_{\sigma(1)} \geq a_{\sigma(2)} \geq  \dots \geq a_{\sigma(n)}$ 
and
$B_i := A_{\sigma(1)} \cup A_{\sigma(2)} \cup \dots \cup A_{\sigma(i)}$. 
\end{remark}

Choquet integral has the following properties: \\
 (1) (monotonicity)  
For any  non-negative measurable functions $f$ and $g$  on $\Omega$, 

$$
0 \leq f \leq g \Rightarrow 0 \leq {\rm (C)} \int f d \mu \leq   {\rm (C)}\int g d \mu .
$$ 
\noindent
(2) (comonotone additivity) 
For any  non-negative measurable functions $f$ and $g$  on $\Omega$, 
if $f$ and $g$ are comonotone, then 
$$
{\rm (C)} \int (f + g)  d \mu
= {\rm (C)} \int f d \mu + {\rm (C)}\int g d \mu ,
$$  
where $f$ and $g$ are {\it comonotone} if 
for any $s,t \in \Omega$ we have  that 
$$
(f(s) -f(t))(g(s)-g(t)) \geq 0, 
$$
that is, 
$$
f(s) < f(t)   \Rightarrow   g(s)  \leq g(t) . 
$$ 

\noindent
(3) (positive homogeneity)
For any  non-negative measurable function $f$ on $\Omega$ and 
any scalar $k \geq 0$, 
$$
 {\rm (C)} \int kf d \mu = k  \Bigl( {\rm (C)} \int f d \mu \Bigr) , 
 $$
where $0 \cdot \infty = 0$.

\begin{remark} \rm 
It is important to note that the Choquet integral can be {\it essentially} characterized as a non-linear monotone positive functional 
which is positively homogeneous and comonotonic additive.  
\end{remark}

\begin{remark} \rm
When $\Omega = {\mathbb R}$, if $f$ and $g$ are monotone increasing, then  $f$ and $g$ are comonotone. 
\end{remark}

For a topological space $X$, let $C(X)$ be the set of all continuous functions on $X$ and 
$C(X)^+$ be the set of all non-negative continuous functions on $X$. More generally, for a $C^*$-algebra $A$, we denote by $A^+$ the set of positive operators of $A$.  

%%%%%%%%%%%%%%%%%%%%%
%% Definition 2.3
%%%%%%%%%%%%%%%%%%%%%
\begin{definition}  \rm Let $A$ be a $C^*$-algebra.  A non-linear positive map 
$\varphi : A^+ \rightarrow  [0, \infty]$ is called a {\it trace} if 
$\varphi$ is unitarily invariant, that is, 
$\varphi(uau^*) = \varphi(a)$ for any $a \in A^+$ and any unitary 
$u$ in the multiplier algebra $M$ of $A$.  
\begin{itemize}
  \item  $\varphi$ is {\it monotone} if 
$a \leq b$ implies  $\varphi(a) \leq  \varphi(b)$ for any $a,b \in  A^+$. 
  \item  $\varphi$ is {\it positively homogeneous} if 
$\varphi(ka) = k\varphi(a)$ for any $a \in A^+$ and any  scalar $k \geq 0$, 
where we regard $0 \cdot \infty = 0$. 
  \item $\varphi$ is {\it comonotonic additive on the spectrum} if 
$$
\varphi(f(a)  + g(a)) = \varphi(f(a)) + \varphi(g(a)) 
$$  
for any $a \in A^+$ and 
any comonotonic functions $f$ and $g$  in $C(\sigma(a))^+$ ( with $f(0) = g(0) = 0$ if $a$ is not invertible), where 
$f(a)$ is a functional calculus of $a$ by $f$ and  $C(\sigma(a))$ is 
the set of continuous functions on the spectrum $\sigma(a)$ of $a$. 
  \item $\varphi$ is {\it monotonic increasing additive on the spectrum} if 
$$
\varphi(f(a)  + g(a)) = \varphi(f(a)) + \varphi(g(a))
$$  
for any $a \in A^+$ and 
any monotone increasing functions $f$ and $g$  in $C(\sigma(a))^+$ (with $f(0)=g(0)=0$ if $a$ is not invertible). 
We may replace $C(\sigma(a))^+$ by $C([0,\infty))^+$, since any monotone increasing function in $C(\sigma(a))^+$ 
can be extended to a monotone increasing function in $C([0,\infty))^+$.  
Then by induction, 
we also have 
$$
 \varphi(\sum_{i=1}^n f_i(a)) = \sum_{i=1}^n \varphi (f_i(a))
% \varphi(f_1)(a) + f_2(a) + \dots +f_n(a)) = \varphi(f_1(a)) + \varphi(f_2(a)) + \dots +\varphi(f_n(a))
$$
for any monotone increasing functions $f_1, f_2, \dots, f_n$ in  $C(\sigma(a))^+$  ( with  $f_i(0) = 0$ for $i=1,2,\ldots, n$ if 
$a$ is not invertible). 
\end{itemize}
\end{definition}

Now we  recall the notions of the generalized singular values (also called generalized $s$-numbers) 
on semifinite von Neumann algebras. 
Murray-von Neumann \cite{murrayvonneumann} started to study the generalized singular values 
for finite factors. The general case of semifinite von Neumann algebras was studied by M. Sonis  \cite{sonis}, 
V. I. Ovchinnikov \cite{ovchinnikov}, F. J. Yeadon \cite{yeason}, and Fack-Kosaki \cite{F-K}.
For the generalized singular values  and the majorization, see also Murray-von Neumann \cite{murrayvonneumann}, 
Kamei \cite{kamei}, Fack \cite{F}, Hiai-Nakamura \cite{H-N2}, Harada-Kosaki \cite{H-K}, \cite{K-S} and we also recommend a book \cite{L-S-Z} by Lord-Sukochev-Zanin on the Calkin correspondence for semi-finite von Neumann algebras.

Let $\M$ be a semifinite von Neumann algebra on a Hilbert space $H$ with a faithful normal semifinite trace $\tau$.  A densely defined closed operator 
$a$ affiliated with $\M$ is said to be {\it $\tau$-measurable} if for each $\epsilon > 0$ there exists a projection $p$ in $\M$ such that  $p(H) \subset D(a)$ and $\tau(I-p) < \epsilon$. Let $a = u|a|$ be the polar decomposition of $a$ and 
$|a| = \int_0^{\infty} \lambda d e_{\lambda}$ be the spectral decomposition of $|a|$. Then $a$ is $\tau$-measurable if and only if 
$\lim _{\lambda  \to \infty} \tau(I- e_{\lambda}) = \lim _{\lambda  \to \infty} \tau(e_{({\lambda}, \infty)} )= 0$.  We denote by  $\tilde{\M}$ 
the set of  $\tau$-measurable operators. For a positive self-adjoint operator $a = \int_0^{\infty} \lambda d e_{\lambda}$ 
affiliated with $\M$, we set 
$$
\tau(a) =  \int_0^{\infty} \lambda d\tau( e_{\lambda}). 
$$
 For $p$ with $1 \leq p \leq \infty$, the non-commutative $L^p$-space $L^p(\M;\tau)$   
is the Banach space consisting of all $a \in \tilde{\M}$ with $||a||_p := \tau(|a|^p)^{1/p} < \infty$, see, for example, \cite{nelson}, \cite{segal}, \cite{terp}, and \cite{F-K}. 
Here  $L^{\infty}(\M;\tau) = \M$. 
We denote by $\tilde{S}$ the closure of $L^1(\M)$ in  $\tilde{\M}$ in the measure topology, where the measure topology on  $\tilde{\M}$ is the topology whose fundamental 
system of neighborhoods around $0$ is given by 
$$
V(\epsilon, \delta) = \{a \in \tilde{\M} | \exists p \in \M : \text{projection s.t.}  
\|ap\| <\epsilon \text{ and }  \tau(I-p) < \delta\}. 
$$
Then $\M$ is dense in  $\tilde{\M}$ in the measure topology.  

%%%%%%%%%%%%%%%%%%%%%
%% Definition 2.7
%%%%%%%%%%%%%%%%%%%%%
\begin{definition} \rm Let $\M$ be a semifinite von Neumann algebra on a Hilbert space $H$ with a faithful normal semifinite trace $\tau$.  
Let  $a$ be a positive self-adjoint operator affiliated with $\M$ with the spectral decomposition $a =  \int_0^{\infty} \lambda d e_{\lambda}(a)$, 
where $e_t = e_{(-\infty,t]}$ is a spectral projection. 
For $t \geq 0$, the ``generalized $t$-th eigenvalue''
$\lambda_t(a)$ of $a$ is defined by 
$$
\lambda_t(a) := \inf \{s \geq 0 \; |\; \tau(e_{(s, \infty)}(a)) \leq t \} , 
$$
where we put $\inf A = + \infty$ if $A$ is an empty set, by convention.  If $a$ is in $M^+$, then $\lambda_0(a) = \|a \|$. 
Furthermore, if $a$ is a  $\tau$-measurable operator and $t > 0$, then $\lambda_t(a) < \infty$ and it is known that 
$\lambda_t(a)$ is also expressed by 
$$
\lambda_t(a) = \inf \{\|ap\| \ |\; \  p \in \M : \text{ projection s.t. } \tau(I-p) \leq t\}. 
$$
We also have a min-max type expression 
$$
\lambda_t(a) = \inf \{\sup _{\xi \in pH, \| \xi \| = 1} \langle a\xi,\xi \rangle  \ | \  p \in \M : \text{ projection s.t. } \tau(I-p) \leq t\}. 
$$
For a general  operator $a$ which is affiliated with $\M$, 
the `` generalized $t$-th singular value''  (i.e. generalized s-number)   
$\mu_t(a)$ of $a$ is defined by 
$\mu_t(a) := \lambda_t(|a|) $.  The map $\mu(a) := [0,\infty) \ni t \mapsto \mu_t(a) \in [0,\infty]$ 
is sometimes called the generalized singular value function of $a$ and satisfies the following properties:
\begin{enumerate}
\item[(1)] The generalized singular value function $\mu(a):t \mapsto \mu_t(a)$ is decreasing and right-continuous. 
\item[(2)] If $0 \leq a \leq b$, then  $\mu_t(a) \leq  \mu_t(b)$. 
\item[(3)] $\mu_{s +t}(a+b) \leq  \mu_s(a)  + \mu_t(b)$. 
\item[(4)] $\mu_{s +t}(ab) \leq  \mu_s(a)  \mu_t(b)$. 
\item[(5)] $\mu_{t}(abc)  \leq \|a\|\mu_{t}(b) \|c\|$.
\item[(6)] $\mu_t(a^*) = \mu_t(a)$. 
\item[(7)] For any unitary $u \in M$, $\mu_t(uau^*) = \mu_t(a)$. 
\item[(8)] $|\mu_t(a) - \mu_t(b)| \leq \|a-b\|$.
\item[(9)] If $a \geq 0$ and $f: [0,\infty) \rightarrow [0,\infty)$ is a continuous and increasing function, then 
$\mu_t(f(a)) = f(\mu_t(a))$. 
\item[(10)] If $\tau (I) = 1$, then $\mu_t(a) = 0$ for all $t >1$. 
\end{enumerate}
See, for example, \cite{F-K}, \cite{fan}, \cite{L-S-Z}, and \cite{ovchinnikov}. 
We note that $\tilde{S}$ is the subspace of $\tilde{\M}$ 
consisting of all operators in $\tilde{\M}$ whose generalized singular value functions vanish at infinity; see \cite{D-P} and 
\cite{D-P-S} more on noncommutative integration theory and operator theory.

For a self-adjont operator $a \in \M$ with a finite trace $\tau$ ,  ``generalized $t$-th eigenvalue''
$\lambda_t(a)$ of $a$  can be  defined by 
$$
\lambda_t(a) := \inf \{s \in \mathbb R \; |\; \tau(e_{(s, \infty)}(a)) \leq t \} , \ \ (t \in [0, \tau(I)).
$$
See \cite{petz} for example, 
although we do not use it in this paper.  For a positive operator $a \in M$,  $\lambda_t(a) = \mu_t(a)$. So 
this does not cause any confusion. 
\end{definition}

\begin{example} \rm If $\M=B(H)$ is a factor of type ${\rm I}_{\infty}$, then $\tilde{\M} = B(H)$ and 
$\tilde{S} = K(H)$.  
\end{example}

\begin{example} \rm Let  $\M=L^{\infty}(\Omega,\mu)$ for some semifinite measure space $(\Omega,\mu)$. 
Then   $\tilde{M}$ consists of measurable functions $f$ on $\Omega$ such that $f$ is bounded except on a set of finite 
measure. In this case, ``generalized $t$-th eigenvalue ``
$\lambda_t(f)$ of positive function $f$ corresponds to the decreasing rearrangement $f^*$ of $f$:
$$
 f^*(t) := \inf \{s \geq 0 \; | \; \mu(\{x  \in \Omega \ | \ |f(x)| > s \}) \leq t \} . 
$$
\end{example}

Now we recall non-linear traces of the Choquet type on finite factor $\M$. 
First,  we assume that $\M$ is a factor of type $I_n$, that is, $\M = M_n({\mathbb C})$

%%%%%%%%%%%%%%%%%%%%%%%%%%%%%%%
%     Definition 
%%%%%%%%%%%%%%%%%%%%%%%%%%%%%%%
\begin{definition} \rm
Let $\alpha: \{0,1,2, \dots, n\} \rightarrow [0, 1]$ be a 
monotone increasing function with $\alpha(0) = 0$ and $\alpha(n) = 1$ .

Let 
$\lambda(a) = (\lambda_1(a),\lambda_2(a),\dots, \lambda_n(a))$ be the list of the 
eigenvalues of $a \in (M_n({\mathbb C}))^+ $ in decreasing order :
$\lambda_1(a) \geq \lambda_2(a) \geq \dots \geq \lambda_n(a)$ with 
counting multiplicities. 
Define 
$\varphi_{\alpha} : (M_n({\mathbb C}))^+ \rightarrow  {\mathbb C}^+$ 
as follows:
\begin{align*}
\varphi_{\alpha}(a) 
& :=  \sum_{i=1} ^{n-1} ( \lambda_i(a)- \lambda_{i+1}(a)) \alpha(i) 
+ \lambda_n (a) \alpha(n)\\
& = \sum_{i=2} ^{n}  \lambda_i(a)( \alpha(i) -  \alpha(i-1)) +  \lambda_1(a)\alpha(1).
\end{align*}
We call $\varphi_{\alpha}$ the non-linear trace of the Choquet type associated with $\alpha$.  
\end{definition}

In the rest of this section, we assume that $\M$ is a factor of type ${\rm II}_1$ with a faithful normal trace 
$\tau$  with $\tau(I) = 1.$ 

%%%%%%%%%%%%%%%%%%%%%%%%%%%%%%%
%     Definition
%%%%%%%%%%%%%%%%%%%%%%%%%%%%%%%
\begin{definition} \rm Let $\M$ be a factor of  type ${\rm II}_1$.
Let $\alpha: [0,1]  \rightarrow [0, 1]$ be a 
monotone increasing function with $\alpha(0) = 0$ and $\alpha(1) = 1$ .   Assume that $\alpha$ is left continuous. 
We think of a weighted dimension function $p \mapsto \alpha(\tau(p))$ for any projection $p \in \M$ as an analog of 
a monotone measure.  
Define a non-linear trace $\varphi_{\alpha} : \M^+ \rightarrow  [0, \infty)$ of the Choquet type associated with $\alpha$ 
as follows: For $a \in \M^+$ with the spectral decomposition $a = \int_0^{\infty} \lambda d e_{\lambda}(a)$,  let  
$$
\varphi_{\alpha}(a) : = \int_0^{\infty}  \alpha(\tau(e_{(s,\infty)}(a)))ds. 
$$
Then for any unitary $u \in \M$, we have that 
$$
\varphi_{\alpha}(uau^*) =  \varphi_{\alpha}(a). 
$$
\end{definition}

\begin{remark} \rm
As shown in \cite{nagisawatatani4},  $\varphi_{\alpha}$ is monotone and sequentially normal in the sense that for any increasing  sequence 
$(a_n)_n$ in $\M^+$, if $ a_n \nearrow a \in \M^+$ in the strong operator topology, then $\varphi_{\alpha}(a_n)  \nearrow \varphi_{\alpha}(a) $. 

For  $a \in \M^+$, 
$$
 \varphi_{\alpha}(a) := \int_0^{\infty}  \alpha(\tau(e_{(s,\infty)}(a)))ds  
 =  \int_0^{\infty}  \alpha(\tau(e_{[s,\infty)}(a)))ds .
$$
For example,  if $a$ has a finite spectrum and 
$a = \sum_{i=1}^{n} a_ip_i$ is the spectral decomposition with $a_1 \geq a_2 \geq \dots \geq a_n > 0$, then 
$$
\varphi_{\alpha}(a) =  \sum_{i=1} ^{n-1} (a_i- a_{i+1})\alpha(\tau(p_1 + \dots + p_i)) +
 a_n \alpha(\tau(p_1 + \dots + p_n)). 
$$
\end{remark}
We can express non-linear traces of the Choquet type using the ``generalized $t$-th eigenvalues''. 
%%%%%%%%%%%%%%%%%%%%%%%%%%%%%%%
%     Proposition 
%%%%%%%%%%%%%%%%%%%%%%%%%%%%%%%
\begin{proposition} {\rm (}\cite[Proposition 2.24]{nagisawatatani4}{\rm )}
\label{prop:Stieltjes}
Let $\M$ be a factor of type ${\rm II}_1$ 
with a faithful normal finite trace $\tau$ with $\tau(I) = 1$, and 
let $\alpha: [0,1]  \rightarrow [0, 1]$ be a 
monotone increasing left continuous function with $\alpha(0) = 0$ and $\alpha(I) = 1$ .
Then the non-linear trace $\varphi_{\alpha}$ is also described as follows: Let $\nu_{\alpha}$ be the Lebesgue-Stieltjes measure 
associated with $\alpha$ such that $\nu_{\alpha}([a,b)) = \alpha(b) - \alpha(a)$ .
For any $a \in \M^+$ we have that 
$$
\varphi_{\alpha}(a) 
 := \int_0^{\infty}  \alpha(\tau(e_{(s,\infty)}(a)))ds = \int_0^1 \lambda_t(a) d\nu_{\alpha}(t). 
$$
\end{proposition}
\begin{remark} \rm
Since  $\M$ be a factor of type ${\rm II}_1$,  $\varphi_{\alpha}$ is operator norm continuous on $\M^+$.  In fact, 
for each $t \in [0,1]$, $a \mapsto  \lambda_t(a)$ is operator norm continuous and 
$|| \lambda_t(a) ||  \leq ||a||$. 
Assume that $\|a_n - a\| \rightarrow 0$ 
for $a_n, a \in \M^+$. Since  $\nu_{\alpha}$ is a probability measure, bounded convergence theorem implies 
that $\varphi_{\alpha}(a_n) = 
\int_0^1 \lambda_t(a_n) d\nu_{\alpha}(t)$ converges to 
$\int_0^1 \lambda_t(a) d\nu_{\alpha}(t) = \varphi_{\alpha}(a)$.  
\end{remark}

%%%%%%%%%%%%%%%%%%%%%%%%
%   Example 2.25
%%%%%%%%%%%%%%%%%%%%%%%%
\begin{example} \rm
Let $\alpha(x) = x$; then we can recover the trace $\tau$,  and for any $a \in \M^+$, 
$$
\varphi_{\alpha}(a) 
 = \int_0^{\infty}  \tau(e_{(s,\infty)}(a))ds = \int_0^{\infty} \lambda_t(a) dt =\tau(a). 
$$
\end{example}

\begin{example} \rm
Let  $$
\alpha(x) 
= \begin{cases}
1,   & (0 < x) \\
0,    & (x = 0); 
\end{cases}
$$
then we can recover the operator norm and for any $a \in \M^+$,  we have 
$$
\varphi_{\alpha}(a) 
 = \int_0^{\infty}  \alpha(\tau(e_{(s,\infty)}(a)))ds = \int_0^{1} \lambda_t(a) d\nu_{\alpha}(t) = ||a||.
$$
\end{example}

\begin{example} \rm
Fix $t > 0$ and 
let  $$
\alpha(x) 
= \begin{cases}
x,    & (0 \leq x \leq t ) \\
t,    & (t < x ); 
\end{cases}
$$
then we obtain a continuous analog of  the Ky Fan norm and for any $a \in \M^+$,  we have 
$$
\varphi_{\alpha}(a) 
 = \int_0^{\infty}  \alpha(\tau(e_{(s,\infty)}(a)))ds = \int_0^{t} \lambda_t(a) dt.   
$$
\end{example}

\begin{remark} \rm 
We can characterize non-linear traces of the Choquet type on a finite factor as in \cite[Theorem 2.29
]{nagisawatatani4}.  
Let $\M$ be a factor of type ${\rm II}_1$ with a faithful normal finite trace with $\tau(1) = 1$. 
Let $\varphi : \M^+ \rightarrow  [0, \infty)$ be a non-linear 
positive map.  Then the following conditions are equivalent:
\begin{enumerate}
\item[$(1)$]  $\varphi$ is a non-linear trace $\varphi = \varphi_{\alpha}$  of the Choquet type associated with  
a monotone increasing and left continuous function $\alpha: [0, 1] \rightarrow [0, 1]$  with $\alpha(0) = 0$ and 
$\alpha(1) = 1$
\item[$(2)$] $\varphi $ is monotonic increasing additive on the spectrum, unitarily invariant, monotone, and 
positively homogeneous.  Moreover $\varphi(I) = 1$.  
\end{enumerate}
\end{remark}

The following fact is important to study the law of large numbers for non-linear traces of the Choquet type on finite factors $\M$.
\begin{theorem} \cite[Remark 2.32
]{nagisawatatani4}
\label{thm:triangle inequality}
Let $\M$ be a factor of type ${\rm II}_1$ with a faithful normal finite trace with $\tau(1) = 1$. 
Let $\varphi_{\alpha}$  be a non-linear trace of the Choquet type associated with  
a monotone increasing and continuous function $\alpha: [0, 1] \rightarrow [0, 1]$  with $\alpha(0) = 0$ and 
$\alpha(1) = 1$.  Assume that $p \geq 1$.  Put $|||a|||_{\alpha,p}(a) := \varphi_\alpha(|a|^p)^{1/p} $. Then the following conditions are equivalent:
\begin{enumerate}
  \item[(a)]  $\alpha$ is concave.
  \item[(b)]  $|||a+b|||_{\alpha, p} \le |||a|||_{\alpha, p} + |||b|||_{\alpha, p}$  for any  $a, b\in \M$ 
\end{enumerate}
\end{theorem}

Finally, we remark that we have Chebyshev's inequality for non-linear traces of the Choquet type.  
\begin{proposition}
Let $\M$ be a factor of type ${\rm II}_1$ with a faithful normal finite trace with $\tau(1) = 1$. 
Let $\varphi_{\alpha}$  be a non-linear trace of the Choquet type associated with  
a monotone increasing and continuous function $\alpha: [0, 1] \rightarrow [0, 1]$  with $\alpha(0) = 0$ and 
$\alpha(1) = 1$. Let $a \in \M$ be self-adjoint and $f \in C(\sigma (a))$ be monotone increasing positive function.  
Take a real number $k \in \sigma (a)$ such that $f(k) >0$. Then 
$$
\varphi_{\alpha}(e_{[k,\infty)}(a)) \leq \frac{\varphi_{\alpha}(f(a))}{f(k)}. 
$$
\end{proposition}
\begin{proof} Since 
$$
f(a) = f(a)(e_{[k,\infty)}(a) + e_{[k,\infty)}(a)^{\perp}) \geq f(a)e_{[k,\infty)}(a) \geq f(k)e_{[k,\infty)}(a), 
$$
and $\varphi_{\alpha}$ is monotone, we have that 
$$
\varphi_{\alpha}(f(a)) \geq \varphi_{\alpha}(f(k)e_{[k,\infty)}(a)) = f(k) \varphi_{\alpha}(e_{[k,\infty)}(a)). 
$$
This shows the conclusion.

\end{proof}
%%%%%%%%%%%%%%%%%%%%%%%%%%%%%%%%%%
%%%%%       section 3
%%%%%%%%%%%%%%%%%%%%%%%%%%%%%%%%%%

\section{Three extensions of non-linear traces of the Choquet type to the self-adjoint operators}

The definition of non-linear traces of the Choquet type in Section 2 is restricted to the positive operators.  In this section, 
we introduce three extensions of non-linear traces of the Choquet type to the self-adjoint operators, that is, symmetric 
extension, anti-symmetric extension, and translatable extension, which are analogs of 
symmetric, anti-symmetric, and translatable Choquet integrals.  First, we recall the definitions of these extensions of Choquet integrals.

Let $X$ be a set and ${\mathcal B}$ a $\sigma$-field on $X$.  
Let  $\mu: {\mathcal B} \rightarrow [0, \infty]$ be a monotone measure on $X$.
Let $f$ be a real valued measurable function on $X$.  Put $f_+ := \frac{1}{2} (|f| +f)$ and 
$f_- := \frac{1}{2} (|f| - f)$.  Then $f = f_+ - f_-$. 

%%%%%%%%%%%%%%%%
%%
%%%%%%%%%%%%%%%% 
\begin{definition} \rm (symmetric Choquet integral)
The symmetric Choquet integral of $f$ is defined by 
$$
{\rm (C_s)} \int f d\mu := {\rm (C)} \int f_+ d\mu - {\rm (C)} \int f_- d\mu,  
$$
if not both terms 
on the right-hand side are infinite. 
\end{definition}

%%%%%%%%%%%%%%%%
%%
%%%%%%%%%%%%%%%%
\begin{definition} \rm (anti-symmetric Choquet integral)
Assume that $\mu(X) < +\infty$. Define the dual monotone measure $\overline{\mu}$  by 
$$
\overline{\mu}(B) :=\mu(X) - \mu(B^c)   \ \ \ (B \in \mathcal B) .
$$

The anti-symmetric Choquet integral of $f$ is defined by 
$$
{\rm (C_a)} \int f d\mu := {\rm (C)} \int f_+ d\mu - {\rm (C)} \int f_- d\overline{\mu}.
$$
\end{definition}

%%%%%%%%%%%%%%%%
%%
%%%%%%%%%%%%%%%%
\begin{definition} \rm (translatable Choquet integral)
Assume that $\mu(X) < +\infty$. Suppose that there exists a real constant $b$ such that 
$f(x) \geq b$ for all $x \in X$. 
Then the translatable Choquet integral of $f$ is defined by 
$$
{\rm (C_t)} \int f d\mu := {\rm (C)} \int (f - b) d\mu + b\mu(X) .
$$
\end{definition}

This definition does not depend on the choice of such $b$. We have translatability, that is,  for any real number $c$, we have that 
$$
{\rm (C_t)} \int (f + c) d\mu := {\rm (C_t)} \int f d\mu + c\mu(X).
$$
Moreover, translatable Choquet integral is equal to anti-symmetric Choquet integral.

Now we shall consider the analog of these facts to three extensions of non-linear traces of the Choquet type.  
Let $\M$ be a factor of type ${\rm II}_1$ with a faithful normal finite trace $\tau$ with $\tau(I) = 1$, and 
let $\alpha: [0,1]  \rightarrow [0, 1]$ be a 
monotone increasing continuous function with $\alpha(0) = 0$ and $\alpha(1) = 1$ .
Let ${\M}_{s.a.}$ be the set of all self-adjoint operators.  For $a \in {\M}_{s.a.}$, put 
$a_+ := \frac{1}{2} (|a| +a) \in \M_+$ and $a_- := \frac{1}{2} (|a| -a) \in \M_+$.  
Then $a = a_+ - a_-$ and $a_+ a_- = 0$. 

%%%%%%%%%%%%%%%%
%%
%%%%%%%%%%%%%%%%
\begin{definition} \rm (symmetric non-linear trace of the Choquet type)
The symmetric non-linear trace of the Choquet type is defined by 
$$
\varphi_{\alpha}^{s}(a) := \varphi_{\alpha}(a_+) - \varphi_{\alpha}(a_-) .
$$
\end{definition}

%%%%%%%%%%%%%%%%
%%
%%%%%%%%%%%%%%%%
\begin{definition} \rm (anti-symmetric non-linear trace of the Choquet type)
Define the dual weight $\overline{\alpha}: [0,1]  \rightarrow [0, \infty)$ of $\alpha$ by 
$$
\overline{\alpha}(t) = \alpha(1) -\alpha(1-t) = 1 - \alpha(1-t)  \quad (t \in  [0,1] ).
$$
Then $\overline{\alpha}$ is a monotone increasing continuous function with 
$\overline{\alpha}(0) = 0$, $\overline{\alpha}(1) = 1$ and $\overline{\overline{\alpha}} = \alpha$.  
Then the anti-symmetric non-linear trace of the Choquet type is defined by 
$$
\varphi_{\alpha}^{a}(a) := \varphi_{\alpha}(a_+) - \varphi_{\overline{\alpha}}(a_-).
$$
\end{definition}

%%%%%%%%%%%%%%%%
%%
%%%%%%%%%%%%%%%%
\begin{example} 
\begin{enumerate}    
\item If $\alpha(t) = t$, then $\overline{\alpha}= \alpha$ and 
$\varphi_{\alpha}^{s} = \varphi_{\alpha}^{a} = \tau$ is the usual trace. 
\item If $\alpha(t) = \sqrt{t}$, then $\overline{\alpha}(t) = 1 -  \sqrt{1-t}$ and 
$\varphi_{\alpha}^{s} \not= \varphi_{\alpha}^{a}$ in general.  
\end{enumerate}
\end{example}

%%%%%%%%%%%%%%%%
%%
%%%%%%%%%%%%%%%%
\begin{definition} \rm (translatable non-linear trace of the Choquet type)
Take a constant real number $c$ such that $a + cI \geq 0$.  
Then translatable non-linear trace of the Choquet type is defined by 
$$
\varphi_{\alpha}^{t}(a) := \varphi_{\alpha}(a + cI) - c.
$$
This definition does not depend on the choice of such $c$.  
In fact, for a positive constant $k$, 
$$
\varphi_{\alpha}(kI) = \int_0^1 \lambda_t(kI) d\nu_{\alpha}(t) =\int_0^1 k d\nu_{\alpha}(t) = k(\alpha(1)- \alpha(0)) =k. 
$$
If $a \geq 0$, then for any constant number $c \geq 0$, we have that 
$$
\varphi_{\alpha}(a + cI ) = \int_0^1 \lambda_t(a + cI) d\nu_{\alpha}(t) 
=\int_0^1  (\lambda_t(a )+ c) d\nu_{\alpha}(t) = \varphi_{\alpha}(a) + c.
$$
For general $a \in M_{s.a.}$, take real numbers $c_1 \geq c_2$ such that 
$a + c_1I \geq 0$ and $a + c_2I \geq 0$, then 
$$
a + c_1I = a + c_2I + (c_1 -c_2)I.  
$$
Therefore 
$$
 \varphi_{\alpha}(a + c_2I) - c_2 = \varphi_{\alpha}(a + c_2I + (c_1 -c_2)I) - (c_1 -c_2) - c_2
  = \varphi_{\alpha}(a + c_1I) - c_1 . 
$$
\end{definition}

Although the symmetric non-linear trace $\varphi_{\alpha}^{s}$ of the Choquet type is symmetric in the sense that 
$$
\varphi_{\alpha}^{s}(-a) = - \varphi_{\alpha}^{s}(a),  \quad (a \in  \M_{s.a.}), 
$$
but it loses translatability, which is an important property we need.  On the other hand, 
the translatable non-linear trace $\varphi_{\alpha}^{t}$ of the Choquet type has many useful properties as shown in the below. Therefore, we usually prefer 
the translatable non-linear trace rather than the symmetric non-linear trace. 

%%%%%%%%%%%%%%%%%%%%
%%
%%%%%%%%%%%%%%%%%%%%
\begin{proposition} 
\label{prop:tranlatable}
Let $\M$ be a factor of type ${\rm II}_1$ with a normalized trace $\tau$.  
Let $\alpha: [0,1]  \rightarrow [0, 1]$ be a 
monotone increasing continuous function with $\alpha(0) = 0$ and $\alpha(1) = 1$ 
Then the translatable non-linear trace $\varphi_{\alpha}^{t}$ of the Choquet type on $\M_{s.a.}$ satisfies the following:
\begin{enumerate}
\item[(1)] {\rm (positively homogeneity)} $\varphi_{\alpha}^{t}(ka) = k \varphi_{\alpha}^{t}(a)$ for any 
$a \in  {\M}_{s.a.}$ and any scalar $k \geq 0$. \\
\item[(2)] {\rm (translatability)}  For any $a \in  {\M}_{s.a.}$ and any real number $c$ 
$$
\varphi_{\alpha}^{t}(a + cI) = \varphi_{\alpha}^{t}(a) + c.
$$
\item[(3)] {\rm (monotonicity)} For $a, b \in {\M}_{s.a.}$, if $a \leq b$, then $\varphi_{\alpha}^{t}(a) \leq 
\varphi_{\alpha}^{t}(b). $\\
\item[(4)] {\rm (unitary invariance)} For any $a \in  \M_{s.a.}$ and any unitary $u \in \M$, we have that 
$\varphi_{\alpha}^{t}(uau^*) =  \varphi_{\alpha}^{t}(a)$. \\
\item[(5)] {\rm (triangle inequality)} Suppose that $\alpha$ is concave on $[0,1]$. Then 
for any $a, b \in {\M}_{s.a.}$, 
$$
 \varphi_{\alpha}^{t}(a + b) \leq  \varphi_{\alpha}^{t}(a) + \varphi_{\alpha}^{t}(b).
$$

\item[(6)] {\rm (operator norm continuity)} $\varphi_{\alpha}^{t}$ is operator norm continuous. 
\end{enumerate}
\end{proposition}
\begin{proof}
(1) Suppose that $a \geq 0$. Then for any scalar $k \geq 0$, we have that
$$
\varphi_{\alpha}(ka) = \int_0^1 \lambda_t(ka) d\nu_{\alpha}(t) 
=\int_0^1 k \lambda_t(a) d\nu_{\alpha}(t) = k \varphi_{\alpha}^{t}(a).
$$
Next suppose that $a \in  {\M}_{s.a.}$. Take a real number $c$ such that $a \geq cI$.  Since $ka \geq kcI$, 
$$
\varphi_{\alpha}^{t}(ka) =\varphi_{\alpha}(ka -kcI) + kc = k\varphi_{\alpha}(a -cI) + kc = k\varphi_{\alpha}^{t}(a).
$$

(2) Take a real number $k$ such that $a + cI \geq kI$, $a \geq kI$ and $c \geq k$.  Then
\begin{align*}
\varphi_{\alpha}^{t}(a + cI) & = \varphi_{\alpha}(a + cI -kI) + kI 
        =  \varphi_{\alpha}^t(a)  + (c-k) + k \\
     & = \varphi_{\alpha}^t(a)  + c.
\end{align*}

(3) For $a, b \in {\M}_{s.a.}$, suppose that $a \leq b$.  Take a real number $k$ such that $kI \leq a \leq b$.  
Since $0 \leq a -kI \leq b -kI$ and $\varphi_{\alpha}$ is monotone on $\M_+$, we have that 
$$
 \varphi_{\alpha}^t(a) = \varphi_{\alpha}(a -kI) + k \leq \varphi_{\alpha}(b -kI) + k = \varphi_{\alpha}^t(b).  
$$

(4) Take a constant real number $c$ such that $a + cI \geq 0$.  Then for any unitary $u \in  {\M}$, 
$u(a + cI)u^* = uau^* + cI \geq 0$.  Since $\varphi_{\alpha}$ is unitarily invariant, 
$$
\varphi_{\alpha}^{t}(uau^*) = \varphi_{\alpha}(u(a + cI)u^*) - c 
=  \varphi_{\alpha}^{t}(a + cI) -c = \varphi_{\alpha}^t(a).   
 $$

(5) Suppose that $\alpha$ is concave on $[0,1]$.  By Theorem \ref{thm:triangle inequality}, 
if  $a$ and $b$ are in $\in {\M}_{+}$, we have that 
$ \varphi_{\alpha}(a + b) \leq  \varphi_{\alpha}(a) + \varphi_{\alpha}(b). $  
Now suppose that $a$ and $b$ are  in ${\M}_{s.a.}$.  Take a real number $c$ such that $a \geq cI$ and $b \geq cI$.  
Then  $a + b \geq 2cI$ and 
\begin{align*}
\varphi_{\alpha}^{t}(a + b) & = \varphi_{\alpha}(a + b -2cI) + 2cI = \varphi_{\alpha}(a -cI + b-cI) + 2cI \\
 &\leq  \varphi_{\alpha}(a - cI)  + \varphi_{\alpha}(b - cI) +2c = \varphi_{\alpha}^{t}(a) + \varphi_{\alpha}^{t}(b). 
\end{align*}
 
(6) Assume that $\|a_n - a\| \rightarrow 0$ for $a_n, a \in {\M}_{s.a.}$. 
Then there exists a real number $k$ such that $a_n \geq kI$ and $a \geq kI$ for any $n$. Then 
\begin{align*}
\varphi_{\alpha}^{t}(a_n) - \varphi_{\alpha}^{t}(a) & =  \varphi_{\alpha}(a_n -kI) + k - (\varphi_{\alpha}(a -kI) +k)  \\
 & = \varphi_{\alpha}(a_n -kI) - \varphi_{\alpha}(a -kI) 
\end{align*}
converges to $0$, since $a_n - kI \geq 0$,  $a - kI \geq 0$ and  $\varphi_{\alpha}$ is operator norm continuous on $\M_+$. 
\end{proof}

The anti-symmetric non-linear trace $\varphi_{\alpha}^{a}$ behaves well for $x \mapsto -x$.

%%%%%%%%%%%%%%%%%%%%%%%%%%%
%%  Proposition
%%%%%%%%%%%%%%%%%%%%%%%%%%%
\begin{proposition} \label{prop:anti-symmetric}
Let $\M$ be a factor of type ${\rm II}_1$ with a normalized trace $\tau$.  
Let $\alpha: [0,1]  \rightarrow [0, 1]$ be a 
monotone increasing continuous function with $\alpha(0) = 0$ and $\alpha(1) = 1$ 
Then the anti-symmetric non-linear trace $\varphi_{\alpha}^{a}$ of the Choquet type satisfies the following:
For $x \in  \M_{s.a.}$, 
\[  \varphi_{\alpha}^{a}(-x) = - \varphi_{\overline{\alpha}}^{a}(x).  \]
\end{proposition}
\begin{proof} 
Since  $x \in \M_{s.a.}$ has the decompositions such that $x = x_+ -  x_-$ and 
$-x = x_- -  x_+$ with $x_+  x_- = 0$, we have that 
$$
\varphi_{\alpha}^{a}(-x) = \varphi_{\alpha}(x_-) - \varphi_{\overline{\alpha}}(x_+),  
$$
and 
$$
\varphi_{\overline{\alpha}}^a(x)
=\varphi_{\overline{\alpha}}(x_+) - \varphi_{\overline{\overline{\alpha}}}(x_-)
=\varphi_{\overline{\alpha}}(x_+) - \varphi_{\alpha}(x_-) .
$$
Therefore  $\varphi_{\alpha}^{a}(-x) = - \varphi_{\overline{\alpha}}^{a}(x)$. 
\end{proof}

If $p \in \M_{s.a.} $ is a projection, then $-p \geq -I$. Therefore 
$$
 \varphi_{\alpha}^t(-p) =  \varphi_{\alpha}(-p + I) +(-1) = \alpha(\tau(I-p)) -1 
 = \alpha(1 - \tau(p)) -1 = -\overline{\alpha}(\tau(p)).
$$
On the other hand, by definition, we have that 
$$
\varphi_{\alpha}^{a}(-p) = -\varphi_{\overline{\alpha}}(p)= -\overline{\alpha}(\tau(p)).
$$
Thus  $\varphi_{\alpha}^t(-p)$ coincides with $\varphi_{\alpha}^{a}(-p)$. 

More generally, this holds for any $a \in  \M_{s.a.}$.  

%%%%%%%%%%%%%%%%%%%%%%
%% Proposition
%%%%%%%%%%%%%%%%%%%%%%
\begin{proposition} 
\label{prop:anti-symmetric = translatable}
Let $\M$ be a factor of type ${\rm II}_1$ with a normalized trace $\tau$.  
Let $\alpha: [0,1]  \rightarrow [0, 1]$ be a 
monotone increasing continuous function with $\alpha(0) = 0$ and $\alpha(1) = 1$ 
Then for any $a \in \M_{s.a.}$, we have that 
$$
\varphi_{\alpha}^{t}(a)  = \varphi_{\alpha}^{a}(a)
$$
\end{proposition}
\begin{proof}
Firstly, consider the case that $a = -b$ ($b \geq 0$)
\begin{align*}
\varphi_{\alpha}^{a}(a) & = \varphi_{\alpha}^{a}(0 - b) = - \varphi_{\overline{\alpha}}(b)\\
& = - \int_0^{\infty}  \overline{\alpha}(\tau(e_{(s,\infty)}(b)))ds \\
& = - \int_0^{\infty} 1 - \alpha(1-\tau(e_{(s,\infty)}(b)))  ds \\
& = - \int_0^{\|b\|} 1 - \alpha(\tau(e_{[0,s]}(b)))  ds \\
& = - \|b\| + \int_0^{\|b\|} \alpha(\tau(e_{[0,s]}(b)))  ds \\
& = - \|b\| + \int_0^{\infty} \alpha(\tau(e_{[s,\infty]}(-b + \|b\|)))  ds \\
& = - \|b\| +  \varphi_{\alpha}(-b + \|b\|) = \varphi_{\alpha}^{t}(a) .
\end{align*}
Secondly, consider a general $a \in \M_{s.a.}$.  We have that 
$$
a = a_+ - a_-  \quad (a_+, a_- \in \M_+, a_+ a_- = 0).  
$$
Put $ x := a + \|a_-\|  =  a_+  + (-a_- + \|a_-\|)\geq 0.$ 
Consider two functions on $[0,\infty)$ such that   
$$
f_1(t) :=
\begin{cases}  0, & (0 \leq t \leq \|a_- \|), \\
 t -  \| a_- \|, & (\| a_- \| \leq t ). 
\end{cases}
$$
$$
f_2(t) :=
\begin{cases}  t, & (0 \leq t \leq \| a_- \|), \\
\| a_- \|, & (\| a_- \| \leq t ). 
\end{cases}
$$
Then $f_1$ and $f_2$ are monotone increasing continuous functions and $f_1(t) + f_2(t) = t$.  
Therefore $f_1(x) = a_+$ and $f_2(x) = -a_- + \|a_-\|$ by functional calculus. 
Since $\varphi_{\alpha}$ is monotonic increasing additive on the spectrum 
to get that 
$$
\varphi_{\alpha}(x) = \varphi_{\alpha}(f_1(x)) + \varphi_{\alpha}(f_2(x)).
$$
Hence we have that
\begin{align*}
\varphi_{\alpha}^{a}(a) & = \varphi_{\alpha}^{a}(a_+ - a_-) = \varphi_{\alpha}(a_+)  - \varphi_{\overline{\alpha}}(a_-)\\
& = \varphi_{\alpha}(a_+) - \|a_-\| + \varphi_{\alpha}(-a_- + \|a_-\|)\\
& = \varphi_{\alpha}(a_+)  + \varphi_{\alpha}(-a_- + \|a_-\|) - \|a_-\|\\
& = \varphi_{\alpha}(a_+  + (-a_- + \|a_-\|) ) - \|a_-\| \\
& = \varphi_{\alpha}(a + \|a_-\|) - \|a_-\|  =  \varphi_{\alpha}^{t}(a). 
\end{align*}
\end{proof}

%%%%%%%%%%%%%%%%%%%%%%%%
%% Proposition
%%%%%%%%%%%%%%%%%%%%%%%%
\begin{proposition}
$|\varphi_{\alpha}^{t}(a)| \leq \|a\|$  for any $a \in  \M_{s.a.}$.
\end{proposition}
\begin{proof}
Since $a \leq \|a\|I$, 
$$
\varphi_{\alpha}^{t}(a)\leq  \varphi_{\alpha}^{t}(\|a\|I) = \|a\|\varphi_{\alpha}^{t}(I) =  \|a\|. 
$$
Then
$$
- \varphi_{\alpha}^{t}(a) = \varphi_{\overline{\alpha}}^{t}(-a) \leq \|-a\| = \|a\|. 
$$
Hence $|\varphi_{\alpha}^{t}(a)| \leq \|a\|$. 
\end{proof}

%%%%%%%%%%%%%%%%%%%%%%%
%% Proposition
%%%%%%%%%%%%%%%%%%%%%%%
\begin{proposition} \label{superadditivity}
Let $\M$ be a factor of type ${\rm II}_1$ with a normalized trace $\tau$.  
Let $\alpha: [0,1]  \rightarrow [0, 1]$ be a 
monotone increasing continuous function with $\alpha(0) = 0$ and $\alpha(1) = 1$ 
\begin{enumerate}
\item[$(1)$] 
If $\alpha$ is concave on $[0,1]$, then $\overline{\alpha}$ is convex on $[0,1]$. 
If $\alpha$ is convex on $[0,1]$, then $\overline{\alpha}$ is concave on $[0,1]$. 
\item[$(2)$] 
If $\alpha$ is convex on $[0,1]$, then for any $a, b \in {\M}_{s.a.}$, 
\[  
\varphi_{\alpha}^{t}(a + b) \geq \varphi_{\alpha}^{t}(a) + \varphi_{\alpha}^{t}(b).
 \]
\item[$(3)$] 
If $\alpha$ is concave on $[0,1]$, then  for any $a \in {\M}_{s.a.}$,
\[ 
\varphi_{\overline{\alpha}}^{t}(a)  \leq \varphi_{\alpha}^{t}(a).
\]
If $\alpha$ is convex on $[0,1]$, then  for any $a \in {\M}_{s.a.}$,
\[
\varphi_{\alpha}^{t}(a) \leq \varphi_{\overline{\alpha}}^{t}(a). 
\]
\end{enumerate}
\end{proposition}
\begin{proof}
$(1)$  
It is trivial. \\
$(2)$   Suppose that  $\alpha$ is covex.  Then $\overline{\alpha}$ is concave and 
for any $a, b \in {\M}_{s.a.}$, 
by a triangle inequality in Proposition \ref{prop:tranlatable}, we have that 
$$
\varphi_{\overline{\alpha}}^{t}((-a) +(-b))
 \leq \varphi_{\overline{\alpha}}^{t}(-a) + \varphi_{\overline{\alpha}}^{t}(-b).
$$
For $x \in  \M_{s.a.}$, 
$\varphi_{\overline{\alpha}}^{a}(-x) = - \varphi_{\alpha}^{a}(x) $ 
and  $\varphi_{\alpha}^{a}(x) = \varphi_{\alpha}^{t}(x)$.  
Hence
$$
- \varphi_{\alpha}^{t}(a + b) \leq  -\varphi_{\alpha}^{t}(a)  -\varphi_{\alpha}^{t}(b).
$$
Therefore 
$$
 \varphi_{\alpha}^{t}(a + b) \geq \varphi_{\alpha}^{t}(a) + \varphi_{\alpha}^{t}(b).
$$
$(3)$  Suppose that $\alpha$ is concave on $[0,1]$.  Then for any $s,t \in [0,1]$, 
$$
\alpha(s + t) \leq \alpha(s) + \alpha(t).
$$
Therefore for any $t \in [0,1]$, 
\begin{align*}
\alpha(t) - \overline{\alpha}(t) & = \alpha(t)  - (1 - \alpha(1-t)) = \alpha(t) + \alpha(1-t) -1 \\
& \geq  \alpha(t + (1-t)) -1 = \alpha(1) -1 = 0.
\end{align*}
Thus  
$$
\overline{\alpha}(t) \leq  \alpha(t) \text{ for any }  t \in [0,1]. 
$$
Then for $x \in  \M_{s.a.}$, 
$$
\varphi_{\alpha}^{t}(a)= \int_0^{\infty}  \alpha(\tau(e_{(s,\infty)}(a)))ds  \geq 
\int_0^{\infty} {\overline \alpha}(\tau(e_{(s,\infty)}(a)))ds = 
\varphi_{\overline{\alpha}}^{t}(a), 
$$
since any projection $p \in \M$ satisfies that $0 \leq \tau(p) \leq 1$. \\
The case that $\alpha$ is convex on $[0,1]$ is similarly proved.
\end{proof}

\noindent
{\bf Caution.} To simplify notation, in the bellow, we use the same symbol $\varphi_{\alpha}(a)$, $(a \in {\M}_{s.a.})$ for 
the translatable extension $\varphi_{\alpha}^{t}(a)$. Moreover, we define the extension
$\varphi_{\alpha}: \M \rightarrow \Bbb C$ by 
$$
\varphi_{\alpha}(x) = \varphi_{\alpha}(a) + i \varphi_{\alpha}(b)
$$
for $x \in \M$ with $x = a + ib, (a,b \in {\M}_{s.a.})$.  And we also call $\varphi_{\alpha}$ the non-linear trace of the 
Choquet type associated with $\alpha$. 

%%%%%%%%%%%%%%
%%
%%%%%%%%%%%%%%
\begin{remark} 
The statements of this section hold for matrix algebras $M_n(\C)$ by similar proofs.
\end{remark}

%%%%%%%%%%%%%%%%%%%%%%%%%%%%%%%%%%%%%%%%%
%%       Section 4
%%%%%%%%%%%%%%%%%%%%%%%%%%%%%%%%%%%%%%%%%
\section{Law of large numbers for non-linear traces of the Choquet type}

Since the non-linear trace $\varphi_{\alpha}$ of the Choquet type associated with $\alpha$ is not additive, 
averages do not converge in general.  Hence, we study the range of their accumulation points, that is, 
we estimate their limit supremum and limit infimum. Thus 
the law of large numbers of a sequence $(a_n)_n$ in $\M_{s.a.}$ for  $\varphi_{\alpha}$ is 
generally described by the estimation of 
\[  \limsup _{n \to \infty} \varphi_{\alpha}((\frac{a_1 + \dots + a_n}{n})^k) \text{ and }  
     \liminf _{n \to \infty} \varphi_{\alpha}((\frac{a_1 + \dots + a_n}{n})^k), \]
for any fixed $k =1,2,3, \dots$. 

We start with a typical example to investigate the law of large numbers for non-linear traces $\varphi_{\alpha}$ of the Choquet type. 

%%%%%%%%%%%%%%%%%%%%
%% Example 4.1
%%%%%%%%%%%%%%%%%%%%
\begin{example} \rm \label{example:UHF}
Let $A = \otimes_{n \in \N}M_2(\C)$ be the UHF algebra with the unique trace 
$\tau=\otimes_{n\in \N} {\rm tr}_2$.  We denote by $\M$ the AFD factor of type ${\rm II}_1$ 
with the normalized trace $\tau$, which is constructed by the weak operator topology closure of 
the image of GNS representation of $A$ by $\tau$. 
We regard that $A \subset \M$.  
Define $\alpha: [0,1]  \rightarrow [0, 1]$ by $\alpha (t) = {\sqrt t}$, then $\alpha$ is a 
monotone increasing continuous function with $\alpha(0) = 0$ and $\alpha(1) = 1 $. 
Then $\alpha$ is concave.  
Let $\varphi_{\alpha}$ be the non-linear trace of the Choquet type associated with $\alpha$.  
Consider a sequence $(p_n)_n$ of projections in $\M$ defined by 
\[    p_n := \overbrace{I \otimes I \otimes \cdots \otimes I}^{n-1}  \otimes 
    \begin{pmatrix}   1 & 0 \\ 0 & 0  \end{pmatrix}
    \otimes I \otimes \cdots,
\]
where only the $n$-th component is non-trivial. 
We also use the notation
\[   p_n =  \overbrace{I \otimes I \otimes \cdots \otimes I}^{n-1}  \otimes 
    \begin{pmatrix}   1 & 0 \\ 0 & 0  \end{pmatrix}\otimes (\otimes_{n+1}^\infty I).  \]
For any $n \neq m$,  $p_n$ and $p_m$ are unitary equivalent in $\M$. Then $p_np_m = p_mp_n$ and 
$$
\varphi_{\alpha}(p_np_m) = {\alpha}(\tau(p_np_m)) = \sqrt {1/4} 
= {\sqrt {1/2}}{\sqrt {1/2}} = \varphi_{\alpha}(p_n)\varphi_{\alpha}(p_m). 
$$ 
Put $s_n = p_1 + p_2 + \cdots +p_n$. 
%Regarding $s_n$ as a diagonal operator in  $M_{2^n}(\C)$, 
%by induction, 
We can show that $s_n$ has the following form:
\[  s_n = \sum_{k=0}^n  (n-k)q_n(k),  \quad  (\sum_{k=0}^n q_n(k) = I), \]
where $\{q_n(k) : k=0,1,\ldots,n \}$ is a family of orthogonal projections with
$\tau(q_n(k)) =\displaystyle \frac{1}{2^n}\binom{n}{k}$.
In fact, for any $n$, we define $q_n(k)$ $(k=0,1,\ldots, n)$ as follows:
\[   q_n(k) = (\sum_{\substack{i_1+i_2+\cdots+i_n=k \\ i_j= 0 \text{ or } 1} }  e_{i_1}\otimes e_{i_2}\otimes \cdots \otimes e_{i_n} )
   \otimes (\otimes_{n+1}^\infty I), \]
where $e_0 = \begin{pmatrix}   1 & 0 \\ 0 & 0  \end{pmatrix}$ and $e_1 = \begin{pmatrix}   0 & 0 \\ 0 & 1  \end{pmatrix}$.
In particular, we have
\begin{gather*}
   q_n(0) = e_0\otimes \cdots \otimes e_0 \otimes(\otimes_{n+1}^\infty I)  \\
\intertext{ and }
     q_n(n) = e_1\otimes \cdots \otimes e_1 \otimes(\otimes_{n+1}^\infty I) .  
\end{gather*}
It is clear that $\{q_n(k) : k=0,1,\ldots,n \}$ is a family of orthogonal projections with
$\tau(q_n(k)) =\displaystyle \frac{1}{2^n}\binom{n}{k}$.
We put
\[   r_n(k) = \sum_{\substack{i_1+i_2+\cdots+i_n=k \\ i_j= 0 \text{ or } 1} }  e_{i_1}\otimes e_{i_2}\otimes \cdots \otimes e_{i_n} .\]
Then we have  $q_n(k) = r_n(k) \otimes (\otimes_{n+1}^\infty I)$ and
\begin{align*}
   p_n & = \overbrace{(e_0+e_1)\otimes \cdots \otimes (e_0+e_1)}^{n-1} \otimes e_0 \otimes (\otimes_{n+1}^\infty I)  \\
        & = ( \sum_{k=0}^{n-1} r_{n-1}(k) )\otimes e_0 \otimes (\otimes_{n+1}^\infty I).
\end{align*}
We show the relation
\[  s_n= \sum_{k=0}^n (n-k)q_n(k)  \]
by induction.
It is clear that 
\[   s_1 = p_1 = e_0\otimes (\otimes_2^\infty I) = q_1(0) .  \]
We assume that
\[   s_n = \sum_{k=0}^n (n-k)q_n(k) . \]
Then we have
\begin{align*}
  s_{n+1} & = s_n + p_{n+1}  \\
     & = \sum_{k=0}^n (n-k) r_n(k) \otimes  (\otimes_{n+1}^\infty I) 
           +  ( \sum_{k=0}^n r_n(k) )\otimes e_0 \otimes (\otimes_{n+2}^\infty I)  \\
     & = \sum_{k=0}^n (n-k) r_n(k) \otimes (e_0+e_1)\otimes  (\otimes_{n+2}^\infty I) 
           +  ( \sum_{k=0}^n r_n(k) )\otimes e_0 \otimes (\otimes_{n+2}^\infty I)  \\
     & = \Bigl( nr_n(0)\otimes e_0 + r_n(0) \otimes e_0 \\
        & \qquad  + \sum_{k=0}^{n-1} \bigl( (n-k)r_n(k) \otimes e_1 + (n-k-1)r_n(k+1)\otimes e_0 + r_n(k+1)\otimes e_0 \bigr)  \Bigr) \\
     & \qquad \otimes (\otimes_{n+2}^\infty I) \\
     & = \Bigl( (n+1)r_n(0)\otimes e_0    + \sum_{k=0}^{n-1}  (n-k) (r_n(k) \otimes e_1 + r_n(k+1)\otimes e_0)  \Bigr)  \otimes (\otimes_{n+2}^\infty I)  \\
     & = \Bigl( (n+1)r_{n+1}(0) + \sum_{k=0}^{n-1}  (n-k) r_{n+1}(k+1) ) \Bigr)  \otimes (\otimes_{n+2}^\infty I) \\
     & = \sum_{k=0}^{n+1} (n-k+1)q_{n+1}(k). 
\end{align*}
So this relation holds.

Then, by definition, 
\[   \varphi_{\alpha}(\frac{s_n}{n}) = \frac{1}{n} \sum_{k=0}^{n-1} \sqrt{\frac{1}{2^n} \sum_{r=0}^k  \binom{n}{r} }.  \]
To compute the limit as $n \rightarrow \infty$, we use Proposition \ref{prop:Stieltjes} to get that 
$$
\varphi_{\alpha}(\frac{s_n}{n}) = \int_0^1 \lambda_t(\frac{s_n}{n}) d\nu_{\alpha}(t). 
$$
We identify  $(p_n)_n$ as a sequence $(a_n)_n$ of independent and identically distributed random variables of ``coin toss'' on 
the infinite product space $\prod_{n \in {\Bbb N}} \{0,1\}$.  By the classical strong law of large numbers, 
$\frac{1}{n}(a_1 + a_2 + \dots +a_n)$ converges to a constant function $\frac{1}{2}$ almost everywhere.  
By  Proposition 3.2.11 in \cite{D-P-S},  the generalized singular valued function $\lambda_t(\frac{s_n}{n})$ converges to a constant function $\frac{1}{2}$ for almost all $t \in [0,1]$.  
By the dominated convergence theorem, 
\begin{align*}
  \lim_{n \to \infty} \varphi_{\alpha}(\frac{s_n}{n}) 
  & = \int_0^1  \lim_{n \to \infty} \lambda_t(\frac{s_n}{n})  d\nu_{\alpha}(t) 
   = \int_0^1 \frac{1}{2} d\nu_{\alpha}(t) \\
  & = \frac{1}{2} \leq {\sqrt {\frac{1}{2}}} = \varphi_{\alpha}(p_1). 
\end{align*}
Similarly, for any fixed $k \in \N$, 
$$
\lim_{n \to \infty} \varphi_{\alpha}((\frac{s_n}{n})^k) 
= \int_0^1 (\frac{1}{2})^k d\nu_{\alpha}(t) = (\frac{1}{2})^k 
\leq ({\sqrt {\frac{1}{2}}}\ )^k = (\varphi_{\alpha}(p_1))^k.
$$ 
\end{example}

As in the example above, we need an {\it inequality} to formulate the law of large numbers for non-linear 
traces $\varphi_{\alpha}$ of the Choquet type. 
Since averages do not converge in general, we study the range of their accumulation points, that is, 
we estimate their limit supremum and limit infimum. 

%%%%%%%%%%%%%%%%%%%%%%%%%%%%%%%%%%%%%%%
%%   Theorem 4.2
%%%%%%%%%%%%%%%%%%%%%%%%%%%%%%%%%%%%%%%
\begin{theorem} \label{psudo law of large numbers}
Let $\M$ be a factor of type ${\rm II}_1$ with a normalized trace $\tau$.  
Let $\alpha: [0,1]  \rightarrow [0, 1]$ be a monotone increasing continuous function with $\alpha(0) = 0$ 
and $\alpha(1) = 1$, and
$\varphi_{\alpha}$ a non-linear trace of the Choquet type associated with $\alpha$.  
Let $(a_n)_n$ be an operator norm bounded sequence $(a_n)_n$ in ${\M}_{s.a.}$. 
Put $s_n = a_1 + a_2 + \dots +a_n$.  
\begin{enumerate}
\item[(1)] Assume that  $\alpha$ is concave. Fix $k \in  \N$. Suppose that there is a constant 
$C(k) \in {\mathbb R}$ such that  
for any mutually different $i_1, i_2, \dots, i_k \in \N$, 
$$
\varphi_{\alpha}({\rm Re} (a_{i_1}a_{i_2} \dots a_{i_k}) ) \leq C(k).  
$$
Then we have that 
\[  \limsup _{n \to \infty}  \varphi_{\alpha}((\frac{s_n}{n})^k) \leq C(k) .\]
\item[(2)] Assume that  $\alpha$ is convex. Fix $k \in  \N$. Suppose that there is a constant 
$\hat{C}(k) \in {\mathbb R}$ such that  
for any mutually different $i_1, i_2, \dots, i_k \in \N$,   
$$
\varphi_{\alpha}({\rm Re} (a_{i_1}a_{i_2} \dots a_{i_k}) ) \geq \hat{C}(k). 
$$
Then we have that 
\[  \liminf _{n \to \infty}  \varphi_{\alpha}((\frac{s_n}{n})^k) \geq \hat{C}(k) .\]
\end{enumerate}
\end{theorem}
\begin{proof}
Since  $(a_n)_n$ be an operator norm bounded sequence $(a_n)_n$, there exists a constant $K > 0$ such that 
$||a_n||\leq K$ for all $n \in {\Bbb N}.$ 

(1)  Fix $k \in {\Bbb N}$. When $k = 1$, the assertion is clear by a triangle inequality in Proposition \ref{prop:tranlatable}. Hence we may assume that $k \geq 2.$ 
Since $n$ goes to $\infty$, we may assume that $n \geq k$.
Put 
$$
D_k(n) := \{i = (i_1,i_2,\dots,i_k) \in \{1,2,\dots,n\}^k \ | \  i_1,i_2,\dots,i_k \text{  are all distinct } \} .
$$
The cardinality $|D_k(n)|$ of $D_k(n)$ is $\dfrac{n!}{(n-k)!}$. 
Since the cardinality of $D_k(n)^c$ is $|D_k(n)^c|=n^k-|D_k(n)|$, we have
\[  \lim_{n\to\infty} \frac{|D_k(n)|}{n^k} = \lim_{n\to\infty}\frac{n}{n}\frac{n-1}{n}\cdots\frac{n-k+1}{n} = 1  
\text{ and } \lim_{n\to\infty} \frac{|D_k(n)^c|}{n^k}= 0.  \]
% Put $g_k(n) := d_k(n) - n^k $.  Then 
% as a polynomial function of $n$, $\deg  g_k(n) = k-1$.  The cardinality of $(D_k(n))^c$ is 
% $\sharp((D_k(n))^c) = n^k -d_k(n) = -g_k(n)$.  

We also have two inequalities as follows:
$$
|| {\rm Re}(a_{i_1}a_{i_2} \dots a_{i_k}) || \leq ||  a_{i_1}a_{i_2} \dots a_{i_k}    || \leq K^k, 
$$
for any $i_1, i_2, \dots, i_k$.  And by assumption, 
$$
\varphi_{\alpha}({\rm Re} (a_{i_1}a_{i_2} \dots a_{i_k})) \leq C(k), 
$$
for any mutually different $i_1, i_2, \dots, i_k$. 

Since $\sum_{i \in D_k(n)} a_{i_1}a_{i_2}\cdots a_{i_k}$ is self-adjoint, so is
$\sum_{i \in D_k(n)^c} a_{i_1}a_{i_2}\cdots a_{i_k}$.
Using a triangle inequality in Proposition \ref{prop:tranlatable} (5), we have that
\begin{align*}
     & \varphi_{\alpha}((\frac{s_n}{n})^k) = \frac{1}{n^k}  \varphi_{\alpha}((a_1 + a_2 + \cdots + a_n)^k)  \\
  = & \frac{1}{n^k}  \varphi_{\alpha}( \sum_{i \in D_k(n)} a_{i_1}a_{i_2} \cdots a_{i_k} 
         +  \sum_{i \in (D_k(n))^c} a_{i_1}a_{i_2} \cdots a_{i_k} )  \\
  \leq &  \frac{1}{n^k}  \varphi_{\alpha}( \sum_{i \in D_k(n)} a_{i_1}a_{i_2} \cdots a_{i_k} )
         +  \frac{1}{n^k} \varphi_{\alpha} (\sum_{i \in (D_k(n))^c} a_{i_1}a_{i_2} \cdots a_{i_k} ).
\end{align*}
Examine each part.  The first part is 
\begin{align*}
& \frac{1}{n^k}  \varphi_{\alpha}( \sum_{i \in D_k(n)} a_{i_1}a_{i_2} \cdots a_{i_k} ) 
= \frac{1}{n^k}  \varphi_{\alpha}( \sum_{i \in D_k(n)} {\rm Re} (a_{i_1}a_{i_2} \cdots a_{i_k}) ) \\
\leq &  \frac{1}{n^k}  ( \sum_{i \in D_k(n)} \varphi_{\alpha}({\rm Re}(a_{i_1}a_{i_2} \cdots a_{i_k}) ))
  \leq  \frac{1}{n^k}  ( \sum_{i \in D_k(n)} C(k) ) \\
= & \frac{|D_k(n)|}{n^k} C(k) \to C(k) \quad (n\to \infty).
%& = \frac{1}{n^k}(n^k + g_k(n))C(k)= (1 + \frac{g_k(n)}{n^k})C(k)
%   \to C(k)  \quad (n \to \infty),
\end{align*}
% since the degree of $g_k(n)$ as a polynomial of $n$ is $k-1$.  
The second part is 
\begin{align*} 
  &  \frac{1}{n^k} \varphi_{\alpha} (\sum_{i \in (D_k(n))^c} a_{i_1}a_{i_2} \cdots a_{i_k} ) 
      \leq \frac{1}{n^k}  ||\sum_{i \in (D_k(n))^c} a_{i_1}a_{i_2} \cdots a_{i_k} || \\
 \leq & \frac{1}{n^k}  \sum_{i \in (D_k(n))^c} ||a_{i_1}a_{i_2} \dots a_{i_k} || 
      \leq  \frac{1}{n^k}  \sum_{i \in (D_k(n))^c} K^k \\
   = &  \frac{|D_k(n)^c|}{n^k} K^k \to 0 \quad (n \to \infty).
\end{align*} 
Therefore 
\[   \limsup _{n \to \infty}  \varphi_{\alpha}((\frac{s_n}{n})^k) 
     \leq \limsup _{n \to \infty} \frac{|D_k(n)|}{n^k}C(k)    
           + \limsup _{n \to \infty} \frac{|D_k(n)^c|}{n^k}K^k = C(k).  
\]

$(2)$  Use the inequality in Proposition \ref{superadditivity} instead of a triangle inequality in Proposition \ref{prop:tranlatable} (5). Then we get the desired assertion in a similar way.  \\
\end{proof}

\begin{example} \rm
Consider the sequence $(p_n)_n$ in Example \ref{example:UHF}.  \\
Let $\alpha(t) ={\sqrt t}$, then  $\overline{\alpha}(t) = 1 - {\sqrt {1-t}} \leq \alpha(t)$, $\alpha$ is concave and $\overline{\alpha}$ 
is convex.  Put $(a_n)_n = (p_n)_n$.  Then for any $i$, 
$$
\varphi_{\alpha}(a_i) = \varphi_{\alpha}(p_i) = \frac{1}{\sqrt 2} =: C(1), 
$$
and 
$$
\varphi_{\overline{\alpha}}(a_i) = \varphi_{\overline{\alpha}}(p_i) = 1- \sqrt {1-\frac{1}{2} }
= 1 - \frac{1}{\sqrt 2} =:{\hat C}(1).
$$
Fix $k \in  \N$. Then for any mutually different $i_1, i_2, \dots, i_k$, 
$$
\varphi_{\alpha}({\rm Re} (a_{i_1}a_{i_2} \dots a_{i_k}) ) = \alpha(\tau(p_{i_1}p_{i_2} \dots p_{i_k}) = 
\alpha(\frac{1}{2^k}) = (\frac{1}{\sqrt 2})^k =: C(k), 
$$
and  
$$
\varphi_{\overline{\alpha}}({\rm Re} (a_{i_1}a_{i_2} \dots a_{i_k}) ) =  
\overline{\alpha}(\frac{1}{2^k}) =1- \sqrt {1-\frac{1}{2^k} }=:{\hat C}(k), 
$$  
Therefore  we have that
\begin{align*}
& 1- \sqrt {1-\frac{1}{2^k} } \leq \liminf _{n \to \infty}  \varphi_{\overline{\alpha}}((\frac{s_n}{n})^k) 
\leq \liminf _{n \to \infty}  \varphi_{\alpha}((\frac{s_n}{n})^k)\\ 
& \leq  \limsup _{n \to \infty}  \varphi_{\alpha}((\frac{s_n}{n})^k) \leq (\frac{1}{\sqrt 2})^k. 
\end{align*}
In fact, in this case we already know that 
\[
\liminf _{n \to \infty}  \varphi_{\alpha}((\frac{s_n}{n})^k) =  \limsup _{n \to \infty}  \varphi_{\alpha}((\frac{s_n}{n})^k)
= \lim  _{n \to \infty}  \varphi_{\alpha}((\frac{s_n}{n})^k) = (\frac{1}{2})^k. 
\]
\end{example}
\begin{example} \rm
Consider the sequence $(p_n)_n$ in Example \ref{example:UHF}. \\
Let $\alpha(t) ={\sqrt t}$.  Put $(a_n)_n = (-p_n)_n$.  Then for any $i$, 
$$
\varphi_{\alpha}(a_i) = \varphi_{\alpha}(-p_i) = - \overline{\alpha}(\frac{1}{2}) = - (1 - \frac{1}{\sqrt 2}) =: C(1), 
$$
and 
$$
\varphi_{\overline{\alpha}}(a_i) = \varphi_{\overline{\alpha}}(-p_i) = - \alpha(\frac{1}{2}) = - \frac{1}{\sqrt 2}
=:{\hat C}(1).
$$
Fix $k \in  \N$. Then 
for any mutually different $i_1, i_2, \dots, i_k$, 
\begin{align*}
& \varphi_{\alpha}({\rm Re} (a_{i_1}a_{i_2} \dots a_{i_k}) ) 
= \varphi_{\alpha}((-1)^k (p_{i_1}p_{i_2} \dots p_{i_k}) ) \\
& = \begin{cases}
- \overline{\alpha}(\frac{1}{2^k}) = - (1 - \sqrt{1-\frac{1}{2^k}}) =: C(k), \ \ &(k \text{ is odd }), \\
 \alpha(\frac{1}{2^k}) = (\frac{1}{\sqrt{2}})^k  =: C(k),  \ \ &(k \text{ is even }). 
\end{cases}
\end{align*}
and
\begin{align*}
& \varphi_{\overline{\alpha}}({\rm Re} (a_{i_1}a_{i_2} \dots a_{i_k}) ) 
= \varphi_{\overline{\alpha}}((-1)^k (p_{i_1}p_{i_2} \dots p_{i_k}) ) \\
& = \begin{cases}
- \alpha(\frac{1}{2^k}) = - ((\frac{1}{\sqrt{2}})^k ) = (- \frac{1}{\sqrt{2}})^k =: \hat{C}(k), \ \ &(k \text{ is odd }), \\
\overline{\alpha}(\frac{1}{2^k}) = 1 - \sqrt{1-\frac{1}{2^k}}  =:\hat{C}(k),  \ \ &(k \text{ is even }). 
\end{cases}
\end{align*}
Therefore, we have that
\begin{align*}
 &\hat{C}(k) \leq \liminf _{n \to \infty}  \varphi_{\overline{\alpha}}((\frac{s_n}{n})^k) 
\leq \liminf _{n \to \infty}  \varphi_{\alpha}((\frac{s_n}{n})^k) \\ 
&\leq  \limsup _{n \to \infty}  \varphi_{\alpha}((\frac{s_n}{n})^k) \leq C(k). 
\end{align*}
\end{example}
\begin{example} \rm
Consider the sequence $(p_n)_n$ in Example \ref{example:UHF}.  \\
Let $\alpha(t) ={\sqrt t}$.  Put $(a_n)_n = (2p_n - I)_n$.  Then for any $i$, 
\begin{align*}
&\varphi_{\alpha}(a_i) = \varphi_{\alpha}(2p_i - I) = \varphi_{\alpha}(p_i - (I-p_i)) \\
&= \alpha(\tau(p_i)) - \overline{\alpha}(\tau(I -p_i)) = \frac{1}{\sqrt 2} - (1-\frac{1}{\sqrt 2} )  = {\sqrt 2} -1 =:C(1), 
\end{align*}
and 
\begin{align*}
&\varphi_{\overline{\alpha}}(a_i) = \varphi_{\overline{\alpha}}(2p_i - I) = \varphi_{\overline{\alpha}}(p_i - (I-p_i)) \\
&= \overline{\alpha}(\tau(p_i)) - {\alpha}(\tau(I-p_i))
 = (1-\frac{1}{\sqrt 2} ) - \frac{1}{\sqrt 2} = 1 -{\sqrt 2}  =:\hat{C}(1), 
\end{align*}
Fix $k \in  \N$. Then for any mutually different $i_1, i_2, \dots, i_k$, 
\[
(2p_{i_1}-I)(2p_{i_2}-I) \dots (2p_{i_k}-I) = 2q - I
\]
for some projection $q \in \M$ with $\tau(q) = \frac{1}{2}$.  
Therefore 
$$
\varphi_{\alpha}({\rm Re} (a_{i_1}a_{i_2} \dots a_{i_k}) ) = \varphi_{\alpha}(2q - I) = {\sqrt 2} -1 =:C(k)
$$
Similarly, we have 
$$
\varphi_{\overline{\alpha}}({\rm Re} (a_{i_1}a_{i_2} \dots a_{i_k}) ) = 1 -{\sqrt 2}  =:\hat{C}(k).
$$
Therefore  we have that
\begin{align*}
& 1- \sqrt {2} \leq \liminf _{n \to \infty}  \varphi_{\overline{\alpha}}((\frac{s_n}{n})^k) 
\leq \liminf _{n \to \infty}  \varphi_{\alpha}((\frac{s_n}{n})^k)\\ 
& \leq  \limsup _{n \to \infty}  \varphi_{\alpha}((\frac{s_n}{n})^k) \leq {\sqrt 2} -1 . 
\end{align*}
In fact, in this case we know that 
\[
\liminf _{n \to \infty}  \varphi_{\alpha}((\frac{s_n}{n})^k) =  \limsup _{n \to \infty}  \varphi_{\alpha}((\frac{s_n}{n})^k)
= \lim  _{n \to \infty}  \varphi_{\alpha}((\frac{s_n}{n})^k) = 0^k = 0, 
\]
since $\tau(2p_i - I) = 0$. 
\end{example}

\begin{example} \rm
Consider the sequence $(p_n)_n$ in Example \ref{example:UHF}.  \\
Let $\beta (t) =t^2$, then  $\overline{\beta}(t) = 1 - (1-t)^2 = 2t -t^2 =:\alpha(t)$ and $\overline{\alpha}(t) = \beta (t)$. 
Moreover $\beta$ is convex, $\overline{\beta}$ is concave and $\beta (t) \leq \overline{\beta}(t)$.  
Put $(a_n)_n = (p_n)_n$.  Then for any $i$, 
$$
\varphi_{\beta}(a_i) = \varphi_{\beta}(p_i) = (\frac{1}{2})^2 = \frac{1}{4} =: \hat{C}(1), 
$$
$$
\varphi_{\overline{\beta}}(a_i) = \varphi_{\overline{\beta}}(p_i) 
= 2(\frac{1}{2}) - (\frac{1}{2})^2 = \frac{3}{4} =:C(1).
$$
Fix $k \in  \N$. Then for any mutually different $i_1, i_2, \dots, i_k$, 
$$
\varphi_{\beta}({\rm Re} (a_{i_1}a_{i_2} \dots a_{i_k}) ) = \beta(\tau(p_{i_1}p_{i_2} \dots p_{i_k}) = 
\beta(\frac{1}{2^k}) = (\frac{1}{4})^k =: \hat{C(}k), 
$$
and  
$$
\varphi_{\overline{\beta}}({\rm Re} (a_{i_1}a_{i_2} \dots a_{i_k}) ) =  
\overline{\beta}(\frac{1}{2^k}) =2(\frac{1}{2^k}) - (\frac{1}{2^k})^2=\frac{1}{2^{k-1}} +  \frac{1}{4^k}=:C(k). 
$$  
Therefore,  we have that
\begin{align*}
 &(\frac{1}{4})^k \leq \liminf _{n \to \infty}  \varphi_{\beta}((\frac{s_n}{n})^k) 
\leq \limsup _{n \to \infty}  \varphi_{\beta}((\frac{s_n}{n})^k)\\ 
& \leq  \limsup _{n \to \infty}  \varphi_{\overline{\beta}}((\frac{s_n}{n})^k) \leq \frac{1}{2^{k-1}} +  \frac{1}{4^k}. 
\end{align*}
In fact, in this case, we know that 
\[
\liminf _{n \to \infty}  \varphi_{\beta}((\frac{s_n}{n})^k) =  \limsup _{n \to \infty}  \varphi_{\beta}((\frac{s_n}{n})^k)
= \lim  _{n \to \infty}  \varphi_{\beta}((\frac{s_n}{n})^k) = (\frac{1}{2})^k. 
\]
\end{example}

Next, we consider sequences of noncommutative self-adjoint operators.  Recall that  in \cite{powers} Powers studied shifts $\sigma$  
on the AFD factor $M$ of type ${\rm II}_1$ such that there is a sequence $(u_n)_{n=0}^\infty$ of self-adjoint unitaries
in $M$ which pairwise either commute or anticommute and $\sigma(u_i) = u_{i+1}$. 

\begin{example} \rm (uniform norm law of large numbers)
Let $A = \otimes_{n \in \N}M_2(\C)$ be the UHF algebra with the unique trace 
$\tau=\otimes_{n\in \N} {\rm tr}_2$.  and  $\M$ the AFD factor of type ${\rm II}_1$ 
with the normalized trace $\tau$.
We study a sequence $(u_n)_{n=0}^\infty$ of self-adjoint unitaries in $A$ satisfying 
\[  u_n=u_n^* = u_n^{-1},  u_iu_j= (-1)^{a(|i-j|)}u_ju_i ,  \]
where $a(n) \in \{0,1\}$ and $a(0)=0$.

First, we consider the case that $a(i)=1$ for any $i\ge 1$, that is, $u_iu_j=-u_ju_i$ for any $i,j$ with $i \not=j$. 
We can construct such a sequence concretely in the following way:
For $ x = \begin{pmatrix} 1 & 0 \\ 0 & -1 \end{pmatrix},  y=\begin{pmatrix}  0 & 1 \\ 1 & 0 \end{pmatrix} \in M_2(\mathbb{C})$,
it follows $x=x^*=x^{-1}$, $y=y^*=y^{-1} \in M_2(\mathbb{C})$ and $xy=-yx$. 
Define a sequence $(u_n)_{n=0}^\infty$ in $A$ as follows:
\begin{align*}
  u_0 & = x\otimes \otimes_2^\infty I_2\\
  u_1 & = y \otimes x \otimes \otimes_3^\infty I_2  \\
  u_2 & = y \otimes y \otimes x \otimes_4^\infty I_2 \\
  u_3 & = y \otimes y \otimes y \otimes x \otimes_5^\infty I_2 \\
       & \cdots \cdots \\
  u_n & = \overbrace{y\otimes y \otimes \cdots \otimes y}^n \otimes x \otimes_{n+2}^\infty I_2 \\
       & \cdots \cdots .
\end{align*}
Then $u_n=u_n^* = u_n^{-1}$ and $u_iu_j= -u_ju_i$ for any $i,j$ with $i\neq j$.

And for all $n$, $\tau(u_n) = 0 =  \varphi_\alpha(u_n)$ for any non-linear trace $\varphi_\alpha$ of the Choquet type. 
These relations imply the following: Put $s_n := u_0+u_1+\cdots +u_{n-1}$. Then $s_n^*= s_n$,  
\[ s_n^*s_n =s_n^2 = (u_0+u_1+\cdots +u_{n-1})^2 = \sum_{i=0}^{n-1}u_i^2 + \sum_{0\le i <j \le n-1}(u_iu_j + u_ju_i) = nI,   \]
and $\|s_n\| = \sqrt{n}$.
So $\frac{1}{\sqrt{n}} s_n$ is a self-adjoint unitary and  we have

\[   \lim_{n\to\infty} \|(\frac{s_n}{n})^k\|= \lim_{n\to \infty} \frac{1}{n^{k/2}} = 0.    \]
\noindent
This shows that a certain strong non-commutativity implies what we call {\it the uniform norm law of large numbers}. 
Moreover, 
\[   \lim_{n\to\infty} \varphi_\alpha((\frac{s_n}{n})^k)=0   \]
by the operator norm continuity of the non-linear trace $\varphi_\alpha$ of the Choquet type.
This means that the limit of the $k$-th moment is $0$.

Suppose that $\alpha$ is concave with  $\alpha(0)=0, \alpha(1)=1$ .
For any mutually different $i_1, i_2, \dots, i_k \in \N$, 
\begin{align*}
{\rm Re} (u_{i_1}u_{i_2} \dots u_{i_k}) 
& = \frac{1}{2}(u_{i_1}u_{i_2} \dots u_{i_k} +  u_{i_k}u_{i_{k-1}} \dots u_{i_1}) \\
& = \frac{1}{2}(u_{i_1}u_{i_2} \dots u_{i_k} +  (-1)^{\frac{k(k-1)}{2}}u_{i_1}u_{i_2} \dots u_{i_k}) 
\end{align*}
is $0$ or $u_{i_1}u_{i_2} \dots u_{i_k}$ itself, so that $u_{i_1}u_{i_2} \dots u_{i_k}$ is a self-adjoint unitary ($\not= I$).  
Hence $\tau(u_{i_1}u_{i_2} \dots u_{i_k})= 0$.  Put $C(k):=  2\alpha(\frac{1}{2})-1$.  Then 
$C(k) \geq 0$, since $\alpha$ is concave with $\alpha(0)=0, \alpha(1)=1$.  Hence  
$$
\varphi_{\alpha}({\rm Re} (u_{i_1}u_{i_2} \dots u_{i_k}) ) \leq 2\alpha(\frac{1}{2})-1 =C(k).  
$$
Thus, we can directly check the statement of Theorem \ref{psudo law of large numbers}  (1):
\[   \limsup_{n\to\infty} \varphi_\alpha((\frac{s_n}{n})^k)= \lim_{n\to\infty} \varphi_\alpha((\frac{s_n}{n})^k) 
= 0 \leq C(k). \]

Nextly, consider $p_n := \frac{1 +u_n}{2}$. Then $(p_n)_n$ is a sequence of projections such that 
$$
\tau(p_i) = \frac{1}{2} \ {\text and} \ p_ip_jp_i = \frac{1}{2}p_i
$$
for any $i,j$ with $i\neq j$. 
Since 
\[   \lim_{n\to\infty} \|\frac{p_0 + p_1 + \dots + p_{n-1}}{n} - \frac{1}{2}  I \| =    
\lim_{n\to\infty} \|\frac{u_0 + u_1 + \dots + u_{n-1}}{2n} \|= 0. \]

We also have that 
\[   \lim_{n\to\infty} \|(\frac{p_0 + p_1 + \dots p_{n-1}}{n})^k - (\frac{1}{2})^kI \|= 0. \]

To study the central limit theorem, we need to consider the limit of  
$$
\frac{s_n}{\sqrt n} := \frac{u_0+u_1+\cdots +u_{n-1}}{{\sqrt n}}
$$ when 
$n \rightarrow \infty$.  We remark that the uniform norm limit of $\frac{s_n}{\sqrt n}$ does not exists but 
$$
\varphi_\alpha( (\frac{s_n}{\sqrt{n}})^k) = \begin{cases}
1, \ \ &  (k \text{ is even }), \\
2\alpha(\frac{1}{2})-1,\ \ &  (k \text{ is odd }).
\end{cases}
$$

In fact, since $(\frac{s_n}{\sqrt{n}} )^2 = I$, $\frac{s_n}{\sqrt{n}}$ is a self-adjoint unitary.  If $k$ is even, then 
$\varphi_\alpha( (\frac{s_n}{\sqrt{n}})^k) = \varphi_\alpha(I) = 1$ .  If $k$ is odd, then 
$$
\varphi_\alpha( (\frac{s_n}{\sqrt{n}})^k) = \varphi_\alpha( (\frac{s_n}{\sqrt{n}})) = 2\alpha(\frac{1}{2})-1. 
$$
Thus for any fixed $k$, $\varphi_\alpha( (\frac{s_n}{\sqrt{n}})^k)$ does not depend on $n$, so that  
its limit when $n \rightarrow \infty$ is that constant.  We shall show  that the uniform norm limit of 
$\frac{s_n}{\sqrt n}$ does not exist.

Firstly, we remark that,  for any self-adjoint unitaries $u$ and $v$ such 
that $uv = -vu$ with projections $p :=\frac{u + I}{2}$ and $q := \frac{v + I}{2}$, we have that 
$$
vpv = I-p, \ \ uqu = I-q, \ \ pvp = 0, \ {\text and} \ \ quq = 0.
$$
In fact, $vu = v(2p-I) = -v + 2vp$ and $uv = (2p-I)v = -v + 2pv$.  
Hence 
$$
(-v + 2vp)v + (-v + 2pv)v = (vu)v + (uv)v = (vu +uv)v = 0.
$$  
Then 
$-v^2 +2vpv -v^2 +2pv^2 = 0$, so that $vpv = I - p$. 
%Similarly $uqu = I -q$.  
Since $pvp$ is self-adjoint and $(pvp)^2 = pvpvp = p(I-p)p = 0$, we have that $pvp = 0$. 
%Similarly $quq = 0$.%  

Now put $U_n :=\frac{1}{{\sqrt n}}(u_0+u_1+\cdots +u_{n-1})$ and 
$V_n := \frac{1}{{\sqrt n}}(u_n+u_{n+1}+\cdots +u_{2n-1})$.  Then $U_n$ and $V_n$ are self-adjoint unitaries.  
For any $k \in \{n,n+1, ...,2n-1\}$, we have that $U_nu_k = -u_kU_n$, so that $U_nV_n = -V_nU_n$. 
Put $Q_n = \frac{V_n + I}{2}$.  Then $Q_nV_nQ_n = Q_n$.  Apply the above remark for $u = U_n$ and $v = V_n$.  Then 
$Q_nU_nQ_n = 0$.  Since  
\begin{align*}
& U_{2n} - U_n =   \frac{1}{{\sqrt {2n}}}(u_0+u_1+\cdots +u_{2n-1}) -  \frac{1}{{\sqrt n}}(u_0+u_1+\cdots +u_{n-1})\\ 
&=  (\frac{1}{{\sqrt {2n}}} - \frac{1}{{\sqrt n}})(u_0+u_1+\cdots +u_{n-1}) 
+  \frac{1}{{\sqrt {2n}}}((u_n+u_{n+1}+\cdots +u_{2n-1}) \\
&= (\frac{1}{\sqrt 2} -1)U_n + \frac{1}{\sqrt 2}V_n, 
\end{align*}
$\|Q_n(U_{2n} -U_n)Q_n \| = \| \frac{1}{\sqrt 2}Q_n \|= \frac{1}{\sqrt 2}$.  
Hence $\|U_{2n} - U_n\| \geq \frac{1}{\sqrt 2}$.  This shows that the uniform norm limit of   
$U_n =\frac{1}{{\sqrt n}}(u_0+u_1+\cdots +u_{n-1})$ does not exist.
\end{example} 
More generally, we can show the following: 

\begin{example} 
Consider a sequence $\{u_n\}_n$ of self adjoint unitaries in $A$ satisfies
\[   u_iu_j= (-1)^{a(|i-j|)}u_ju_i \text{ and } a(K l)=1 \quad (l=1,2,3,\ldots)  \]
for some positive integer $K$.
We devide the sequence $\{u_n\}_n$ into $K$ subsequences
\[  \{u_{Kn}\}_n, \; \{u_{Kn+1}\}_n, \ldots, \; \{u_{Kn+K-1} \}_n  .   \]
Each subsequence $\{u_{Kn+k}\}_n$ $(k=0,1,2,\ldots, K-1)$ satisfies 
\[     u_{Ki +k}u_{Kj+k} = - u_{Kj+k}u_{Ki+k}  \text{ for any }  |i-j|\ge 1.  \]
By the above argument
\[   \lim_{n\to\infty} \|\frac{1}{n}(u_k + u_{K+k}+ \ldots + u_{K(n-1)+k}) \|=0, \quad k=0,1,\ldots, K-1,    \]
and
\begin{align*}
 & \frac{s_{Kn}}{Kn}  = \frac{1}{Kn}(u_0+u_1+u_2+\cdots + u_{Kn-1})  \\
=  &  \frac{1}{K} ( \frac{1}{n}(u_0+u_K+ \cdots +u_{K(n-1)}) + \cdots  
 +\frac{1}{n}(u_{K-1} + u_{2K-1} + \cdots + u_{Kn-1})). 
\end{align*}
So we have
\[   \lim_{n\to\infty} \|\frac{s_n}{n}\| = 0,   \lim_{n\to\infty} \|(\frac{s_n}{n})^k\| = 0,  \text{ and }   \lim_{n\to\infty} \varphi_\alpha((\frac{s_n}{n})^k)=0 .  \]
\end{example}

%%Of course, we do {\it not} have uniform norm law of large numbers in general as follows: 

\vspace{10mm}

If the constant $C(k)$ in Theorem \ref{psudo law of large numbers}  (1) can be chosen as $C(k) = m^k$ with 
$$
m:= \varphi_\alpha (a_n)  \quad \text{ for all } n=1,2, \ldots, 
$$
and does not depned on $n$, then $(a_n)_n$ in $\M_{s.a.}$ behaves like {\it independence}.  

\begin{definition} \rm (independence)
Let $\M$ be a factor of type ${\rm II}_1$ with a normalized trace $\tau$.  
Let $\alpha: [0,1]  \rightarrow [0, 1]$ be a 
monotone increasing continuous function with $\alpha(0) = 0$ and $\alpha(1) = 1$ 
Let $\varphi_{\alpha}$  be a non-linear trace of the Choquet type associated with $\alpha$. 
A sequence $(a_n)_n$ in $\M_{s.a.}$ is called {\it independent} with respect to $\varphi_{\alpha}$  
if for any mutually different $i_1, i_2, \ldots, i_m$, 
$$
\varphi_{\alpha}(a_{i_1}a_{i_2} \dots a_{i_m})
= \varphi_{\alpha}(a_{i_1}) \varphi_{\alpha}(a_{i_2})  \dots \varphi_{\alpha}(a_{i_m}). 
$$
A sequence$(a_n)_n$ in ${\M}_{s.a.}$ is called {\it subindependent} with respect to $\varphi_{\alpha}$   if 
for any mutually different $i_1, i_2, \dots, i_m$, $m = 1,2,3,...$, 
$$
\varphi_{\alpha}({\rm Re} (a_{i_1}a_{i_2} \dots a_{i_m}) )
\leq \varphi_{\alpha}(a_{i_1}) \varphi_{\alpha}(a_{i_2})  \dots \varphi_{\alpha}(a_{i_m}). 
$$
A sequence$(a_n)_n$ in ${\M}_{s.a.}$ is called {\it superindependent} with respect to $\varphi_{\alpha}$   if 
for any mutually different $i_1, i_2, \dots, i_m$, $m = 1,2,3,...$, 
$$
\varphi_{\alpha}({\rm Re} (a_{i_1}a_{i_2} \dots a_{i_m})) 
\geq \varphi_{\alpha}(a_{i_1}) \varphi_{\alpha}(a_{i_2})  \dots \varphi_{\alpha}(a_{i_m}). 
$$
For example, the sequence $(p_n)_n$ in Example \ref{example:UHF} is independent, subindependent, and superindependent 
with respect to $\varphi_{\alpha}$. 
\end{definition}

%%%%%%%%%%%%%%%%%%%%%%%%%%%%%%%%%%%%%%%
%%
%%%%%%%%%%%%%%%%%%%%%%%%%%%%%%%%%%%%%%%
\begin{corollary} \label{law of large numbers}
Let $\M$ be a factor of type ${\rm II}_1$ with a normalized trace $\tau$.  
Let $\alpha: [0,1]  \rightarrow [0, 1]$ be a monotone increasing continuous function with $\alpha(0) = 0$ 
and $\alpha(1) = 1$, and
$\varphi_{\alpha}$ a non-linear trace of the Choquet type associated with $\alpha$.  
Let $(a_n)_n$ be an operator norm bounded sequence $(a_n)_n$ in ${\M}_{s.a.}$ with a constant $m \in \R$ such that 
\[   m= \varphi_\alpha (a_n)  \quad \text{ for all } n=1,2, \ldots  \] 
and put $s_n = a_1 + a_2 + \dots +a_n$.
\begin{enumerate}
\item[(1)] If $\alpha$ is concave and the sequence $(a_n)_n$ is subindependent with respect to $\varphi_\alpha$, 
then for any fixed $k \in  \N$, we have that 
\[  \limsup _{n \to \infty}  \varphi_{\alpha}((\frac{s_n}{n})^k) \leq m^k .\]
\item[(2)] If $\alpha$ is convex and the sequence $(a_n)_n$ is superindependent with respect to $\varphi_\alpha$, 
then for any fixed $k \in  \N$, we have that 
\[  \liminf _{n \to \infty}  \varphi_{\alpha}((\frac{s_n}{n})^k) \geq m^k .\]
\item[(3)] Suppose that $\alpha$ is concave, the sequence $(a_n)_n$ is subindependent with respect to $\varphi_\alpha$, 
and the sequence $(a_n)_n$ is superindependent with respect to $\varphi_{\overline \alpha}$ with a constant 
$\overline{m} \in \R$ such that 
\[ {\overline  m} := \varphi_{\overline \alpha} (a_n)  \quad \text{ for all } n=1,2, \ldots  \] 
then for any fixed $k \in  \N$, we have that 
\[
{\overline m}^k \leq \liminf _{n \to \infty} \varphi_{\overline \alpha}((\frac{s_n}{n})^k) 
\leq \liminf _{n \to \infty}  \varphi_{\alpha}((\frac{s_n}{n})^k)  
\leq \limsup _{n \to \infty}  \varphi_{\alpha}((\frac{s_n}{n})^k)  \leq m^k. 
\] 
Moreover any accumulation point of the sequence $\varphi_{\alpha}((\frac{s_n}{n})^k)$ is in the interval 
$[{\overline m}^k, m^k]$. 
\end{enumerate}
\end{corollary}
\begin{proof} (1) and (2) are direct. 
$(3)$ Since $\alpha$ is concave,  $\overline \alpha$ is convex and  
 $\varphi_{\overline \alpha}(a) \leq \varphi_{\alpha}(a)$ for any $a \in {\M}_{s.a.}$, 
 $(1)$ and $(2)$ implies $(3)$. 
\end{proof}

Consider the case that $\alpha(t) = t$ for $t \in [0,1]$.  Then $\varphi_{\alpha}$ become the linear trace $\tau$ on $\M$.  
And as a corollary of the above Corollary \ref{law of large numbers}, we get a usual {\it equality formula} of the 
law of large numbers. 

%%%%%%%%%%%%%%%%%
%%
%%%%%%%%%%%%%%%%%
\begin{corollary}
Let $\M$ be a factor of type ${\rm II}_1$ with a normalized trace $\tau$. 
Let $(a_n)_n$ be an operator norm bounded sequence $(a_n)_n$ in ${\M}_{s.a.}$ and 
put $s_n = a_1 + a_2 + \dots +a_n$.
Supposet that $(a_n)_n$ is independent with respect to $\tau$ and 
$$
\tau(a_1) = \tau(a_2) = \dots   (=:m). 
$$
Then for any fixed $k \in {\Bbb N}$, we have that
$$
\lim_{n \rightarrow \infty}  \tau((\frac{s_n}{n})^k) =  m^k. 
$$
Moreover for any continuous real-valued function $f$ on ${\Bbb R}$, 
we have that 
$$
\lim_{n \rightarrow \infty}  \tau(f(\frac{s_n}{n})) =  f(m). 
$$
\end{corollary}
\begin{proof}
Since $\alpha(t) = t$ for $t \in [0,1]$, ${\overline{\alpha}}(t) = t$ and  
$\varphi_{\alpha} = \varphi_{\overline{\alpha}} = \tau$. It is clear that $\alpha(t) = t$ is concave and convex.   
Applying the above Corollary \ref{law of large numbers}, we have that 
$$
m^k = \overline{m} ^k \leq \liminf _{n \rightarrow \infty}  \varphi_{\alpha}((\frac{s_n}{n})^k) 
\leq \limsup _{n \rightarrow \infty}  \varphi_{\alpha}((\frac{s_n}{n})^k) \leq m^k
$$
Hence 
$$
\liminf _{n \rightarrow \infty}  \tau((\frac{s_n}{n})^k )= \limsup _{n \rightarrow \infty} \tau((\frac{s_n}{n})^k) = \lim_{n \rightarrow \infty}  \tau((\frac{s_n}{n})^k) =  m^k. 
$$
Since $\tau$ is linear, for any polynomial $g$ 
$$
\lim_{n \rightarrow \infty}  \tau(g(\frac{s_n}{n})) =  g(m). 
$$
Moreover, fix a continuous real-valued function $f$ on ${\Bbb R}$.  
Since $(a_n)_n$ is an operator norm bounded sequence $(a_n)_n$ in ${\M}_{s.a.}$, there exists a constant 
$K > 0$ such that for any $n \in {\Bbb N}$ we have that $\|a_n\| \leq K$.  Then $||\frac{s_n}{n}|| \leq K$  and 
any spectrum $\sigma (\frac{s_n}{n})$ is included in $[-K,K]$. And $f$ is uniformly approximated by a polynomial sequence 
$(f_i)_i$ on $[-K,K]$.  Therefore for any $\epsilon > 0$, there exists $i \in {\Bbb N}$ such that
$$
||f -f_i||_{\infty} :=\sup \{|(f - f_i)(x)| ; x \in [-K,K]\} < \frac{\epsilon}{3}.
$$
Then
$$
||f(\frac{s_n}{n}) -f_i(\frac{s_n}{n})|| =||(f-f_i)  (\frac{s_n}{n})|| \leq ||f -f_i||_{\infty}< \frac{\epsilon}{3}.
$$
Since $\M$ is a factor of type ${\rm II}_1$, $\tau$ is operator norm continuous and $||\tau|| =1$. 
Thus 
$$
|\tau(f(\frac{s_n}{n})) - \tau(f_i(\frac{s_n}{n}))| < \frac{\epsilon}{3}.
$$
Since $f_i$ is a polynomial, There exists $n_0 \in {\Bbb N}$ sucn that for any $n \in {\Bbb N}$ if $n \geq n_0$, 
then
$$
|\tau(f_i(\frac{s_n}{n})) - f_i(m)| < \frac{\epsilon}{3}. 
$$
Since $m = \tau(a_1) = \tau(\frac{s_n}{n}) \in [-K,K]$, $|f_i(m) -f(m)| < \frac{\epsilon}{3}.$ 
Therefore, 
\begin{align*}
&  |\tau(f(\frac{s_n}{n})) - f(m)| \\
&\leq |\tau(f(\frac{s_n}{n})) - \tau(f_i(\frac{s_n}{n}))| + 
|\tau(f_i(\frac{s_n}{n})) - f_i(m)| + |f_i(m) -f(m)| < \epsilon, 
\end{align*}     
which proves the desired conclusion.
\end{proof}

For a subindependent positive sequence, we need to restrict to {\it absolutely monotone functions}.  

%%%%%%%%%%%%%%%%%%%%
%%
%%%%%%%%%%%%%%%%%%%%
\begin{definition} \rm
A $C^{\infty}$ function $f$ on a open interval $(a,b)$ is called {\it absolutely monotone} if 
$$
 f^{(n)}(x) \geq 0   \ \ \ \text{ for  any } x \in (a,b) \text{ and for all }  n = 0,1,2,\dots .  
$$ 
Since each  $f^{(n)}$ is increasing on $(a,b)$,  $f^{(n)}$ is extended on $[a,b)$ continuously  by 
putting $f^{(n)} (a) = \lim_{ x \searrow a}  f^{(n)} (x)$. Moreover,  $f$ can be extended to an analytic function  on 
a open disc $\{z \in {\mathbb C} \ |  \ |z -a| < b-a \}$, known as little Bernstein Theorem. 
See, for examle, Widder \cite{widder}.  
\end{definition}

%%%%%%%%%%%%%%%%%%%%
%%
%%%%%%%%%%%%%%%%%%%%
\begin{corollary}
Let $\M$ be a factor of type ${\rm II}_1$ with a normalized trace $\tau$.  
Let $\alpha: [0,1]  \rightarrow [0, 1]$ be a monotone increasing continuous function with $\alpha(0) = 0$ 
and $\alpha(1) = 1$.  
Suppose that $\alpha$ is concave.  
Let $(a_n)_n$ be an operator norm bounded sequence $(a_n)_n$ in ${\M}_{+}$ and 
put $s_n = a_1 + a_2 + \dots +a_n$.
Supposet that $(a_n)_n$ is subindependent with respect to $\varphi_{\alpha}$ and 
$$
\varphi_{\alpha}(a_1) = \varphi_{\alpha}(a_2) = \dots   (=:m)
$$
Then for any continuous function $f$ on $[0,\infty)$ such that $f$ is absolutely monotone  on $(0,\infty)$ , we have that 
$$
\limsup _{n \rightarrow \infty}  \varphi_{\alpha}(f(\frac{s_n}{n})) \leq f(m). 
$$
\end{corollary}
\begin{proof} Since  $(a_n)_n$ be an operator norm bounded sequence $(a_n)_n$, there exists a constant $K > 0$ such that 
$||a_n||\leq K$ for all $n \in {\Bbb N}.$ \\

Firstly, assume that $f$ is a polynomial such that $f(x) = \sum_{k =1} ^i a_k x^k$ for some $a_1 \geq 0, \dots, a_i \geq 0$. 
Then 
$$
\limsup _{n \rightarrow \infty}  \varphi_{\alpha}(f(\frac{s_n}{n})) 
\leq \sum_{k=1} ^i \limsup _{n \rightarrow \infty} a_k \varphi_{\alpha}((\frac{s_n}{n})^k) 
\leq  \sum_{k =1} ^i a_k m^k = f(m) .
$$
Secondly, suppose that $f$ is a continuous function on $[0,\infty)$ such that $f$ is absolutely monotone  on $(0,\infty)$.  
Then 
$$
f(x) = \sum_{k =1} ^{\infty} a_k  x^k \text{ for some } a_k \geq 0, (k = 1,2,\ldots \infty), \quad  x \in [0,K]. 
$$
Hence  $f$ is unformly 
approximated by finite sums 
$$
f_i(x) := \sum_{k =1} ^i a_k x^k \leq f(x) \ \ \ (x \in [0,K]).
$$
Take any $\epsilon > 0$. Then there exists $i \in {\Bbb N}$ such that 
$$
f(x) \leq f_i(x) + \epsilon \ \ \ (x \in [0,K]). 
$$
Since $\limsup _{n \rightarrow \infty}  \varphi_{\alpha}(f_i(\frac{s_n}{n})) \leq f_i(m)$, 
there exists $n_0 \in {\Bbb N}$ such that for all $n \geq n_0$ 
$$
\varphi_{\alpha}(f_i(\frac{s_n}{n})) \leq f_i(m) + \epsilon. 
$$
Since the spectrum $\sigma(\frac{s_n}{n}) \subset [0,K]$, 
$$
f(\frac{s_n}{n}) \leq f_i(\frac{s_n}{n}) + \epsilon I. 
$$
Since $\varphi_{\alpha}$ is monotone and translation invariant, 
\begin{align*}
& \varphi_{\alpha} (f(\frac{s_n}{n}) )  \leq \varphi_{\alpha}( f_i(\frac{s_n}{n}) + \epsilon I) \\
& = \varphi_{\alpha}( f_i(\frac{s_n}{n})) + \epsilon \leq  f_i(m) + \epsilon +  \epsilon \leq f(m) + 2 \epsilon. 
\end{align*}
Then 
$$
\limsup _{n \rightarrow \infty}  \varphi_{\alpha}(f(\frac{s_n}{n})) \leq f(m) + 2 \epsilon.
$$
This implies the statement.
\end{proof}

We give a sufficient condition that the limit $\lim _{n \rightarrow \infty}  \varphi_{\alpha}((\frac{s_n}{n})^k)$ exists, 
which includes the Example \ref{example:UHF}. 

%%%%%%%%%%%%%%%%%%%%%%
%%
%%%%%%%%%%%%%%%%%%%%%%
\begin{proposition} 
Let $\M$ be a factor of type ${\rm II}_1$ with a normalized trace $\tau$.  
Let $\alpha: [0,1]  \rightarrow [0, 1]$ be a monotone increasing continuous function with $\alpha(0) = 0$ 
and $\alpha(1) = 1$.  
Suppose that $\alpha$ is concave.  
Let $\sigma: \M \rightarrow \M$ be a *-endomorphism with $\sigma(I) = I.$
Let $(a_n)_n$ be an operator norm bounded sequence $(a_n)_n$ in ${\M}_+$ satisfying 
$a_n = \sigma^{n-1}(a_1)$ for all $n \in {\Bbb N}$. 
Put $s_n = a_1 + a_2 + \dots +a_n$.
Then for any fixed $k \in {\Bbb N}$, the limit of the sequence $(\varphi_{\alpha}((\frac{s_n}{n})^k))_n$ exists 
and equals to  $\inf_n  \varphi_{\alpha}((\frac{s_n}{n})^k)$.
\end{proposition}
\begin{proof} For $x \in \M$, put $|||x|||_{\alpha,k}:= \varphi_\alpha(|x|^k)^{1/k} $.  Then by 
Theorem \ref{thm:triangle inequality}, this gives a norm on $\M$, since $\alpha$ is concave. 
Since $\tau$ is the unique trace, $\sigma$ preserves the trace $\tau$. Hence 
$|||\sigma(x)|||_{\alpha,k} = |||x|||_{\alpha,k}.$ 
Put 
$$
x_n := |||s_n|||_{\alpha,k} = |||a_1 + a_2 + \dots +a_n|||_{\alpha,k} .
$$
Then 
\begin{align*}
     & x_{n+m}  = |||a_1 + a_2 + \dots +a_{n+m}|||_{\alpha,k} \\
   \leq & |||a_1 + a_2 + \dots +a_n|||_{\alpha,k} + |||a_{n+1} + a_{n+2} + \dots +a_{n+m}|||_{\alpha,k} \\
   = & |||a_1 + a_2 + \dots +a_n|||_{\alpha,k} + |||\sigma^n(a_1 + a_2 + \dots +a_m)|||_{\alpha,k} \\
   = & |||a_1 + a_2 + \dots +a_n|||_{\alpha,k} + |||a_1 + a_2 + \dots +a_m|||_{\alpha,k} = x_n + x_m.
\end{align*}
Thus $(x_n)_n$ satisfies subadditivity. By Fekete's lemma, the limit of the sequence $(\frac{x_n}{n})_n$ exists 
and  equals to  $\inf_n  \frac{x_n}{n}$. 
We note that 
$$
(\frac{x_n}{n})^k =  \varphi_{\alpha}((\frac{s_n}{n})^k) 
$$
and  $t \mapsto t^k$ is continuous and monotone increasing.  
Hence an inequality
$$
\lim_{n \to \infty} \frac{(\varphi_{\alpha}(s_n^k))^{1/k}}{n} 
= \inf_n \frac{(\varphi_{\alpha}(s_n^k))^{1/k}}{n} 
\leq  \frac{(\varphi_{\alpha}(s_n^k))^{1/k}}{n}
$$
implies that 
$$
\lim_{n \to \infty} \frac{\varphi_{\alpha}(s_n^k)}{n^k}
 \leq  \frac{ \varphi_{\alpha}(s_n^k)}{n^k}
$$
for any $n \in {\Bbb N}$.  Therefore 
$$
\lim_{n \to \infty} \varphi_{\alpha}((\frac{s_n}{n})^k)
 \leq  \inf_n  \varphi_{\alpha}((\frac{s_n}{n})^k).
$$
This implies the statement.
\end{proof}


\begin{thebibliography}{99}
%%%%%%%%%%%%

\bibitem{A-C}
T. Ando and M. Choi, 
\textit{Non-linear completely positive maps}, 
in Aspects of Positivity in Functional Analysis, R. Nagel et al. eds., 
North-Holland, Amsterdam, 1986, 3-13.
%%%%%%%%%%%%%%%%%
\bibitem{Ar1}
W. B. Arveson, 
\textit{Subalgebras of $C^*$-algebras},
Acta Math., 123 (1969), 141-224. 
%%%%%%%%%%
\bibitem{Ar2}
W. B. Arveson, 
\textit{Nonlinear states on $C^*$-algebras},
Operator Algebras and Mathematical Physics, Contemporary Math., 
Vol 62, Amer. Math. Soc., 1987, 283-343. 
%%%%%%%%%%%%%
\bibitem{B-N}
D. Bel\c{t}it\u{a} and K-H. Neeb, 
\textit{Nonlinear completely positive maps and dilation theory for real involutive algebras},
Integral Equations and Operator Theory 83 (2015), 517-562.
%%%%%%%%%%%%%
%%%%%%%%%
\bibitem{bhatia1}
R. Bhatia,
\textit{Matrix Analysis},
Graduate Texts in Mathematics 169,
Springer-Verlag New York, 1997.
%%%%%%%%%
\bibitem{bhatia2}
R. Bhatia,
\textit{Positive Definite Matrices},
Princeton University Press, 2007.
%\bibitem{blackadar}
%B. Blackadar,
%\textit{Operator Algebras},
%Encyclopaedia of Mathematical Sciences, 122. Operator Algebras and Non-commutative G%eometry, 3. Springer-Verlag, Berlin, 2006.

%%%%%%%%%%%%%%%%%%%

\bibitem{C-M-R}
A. Chateauneuf, M. Grabisch and A. Rico, 
\textit{Modeling attitudes toward uncertainty through the use of the Sugeno integral}, 
J. Math. Econom., 44 (2008), 1084--1099. 
%%%%%%%%%%%%%%%%%
\bibitem{chen}
Z. Chen, 
\textit{Strong laws of large numbers for sublinear expectations}, 
Sci.China Math. 59 (2016), 945-954.
%%%%%%%%%%%%%%%%%%%%%%
\bibitem{Ch}
G. Choquet, 
\textit{Theory of capacities},
Ann. Inst. Fourier 5 (1953), 131--295. 
%%%%%%%%%%%%%
\bibitem{D-M}
A. Dadkhah and M. Moslehian,
\textit{Non-linear positive maps between $C^*$-algebras},
Linear Multilinear Algebra, 68(2020), 1501--1517.
%%%%%%%%%%%%%
\bibitem{D-M-K}
A. Dadkhah, M. Moslehian, and N. Kian, 
\textit{Continuity of positive non-linear positive maps between $C^*$-algebras}, 
Studia Math. 263 (2022), no. 3, 241-266.
%%%%%%%%%%%%%
\bibitem{C-B}
L. M. de Campos and M. J. Bola\~{n}os, 
\textit{Characterization and comparison of Sugeno and Choquet integrals}, 
Fuzzy Set and Systems 52 (1992), 61-67.
%%%%%%%%%%%%%%%%%%%

\bibitem{De}
C. Dellacherie, 
\textit{Quelques commentaires sur les prolongenments de capacit\'{e}s}, 
In: S\'{e}minaire de probabilit\'{e}s. V, Strasbourg, 1969-1970, Lecture Notes in Math., 
191, Springer, 1971, pp.77-81.
%%%%%%%%%%%%%%%%%%%
\bibitem{D-G}
D. Denneberg and M. Grabisch, 
\textit{Measure and integral with purely ordinal scales}, 
J. Math. Psych., 48 (2004), 15-27. 
%%%%%%%%%%%%%%%%%%%%
\bibitem{dobrakov}
I. Dobrakov, 
\textit{On submeasures I}, 
Dissertationes Math. (Rozprawy Mat.), 112 (1974), 1-35. 
%%%%%%%%%%%%%%%%%%%%%%%%
\bibitem{D-P}
P.G. Dodds, B. de Pagter, 
\textit{Normed K\"{o}the spaces: A non-commutative viewpoint}, 
Indag. Math. (N.S.) 25 (2014), 206-249.
%%%%%%%%%
\bibitem{D-P-S}
P.G. Dodds, B. de Pagter, and F. A. Sukochev, 
\textit{Noncommutative Integration and Operator Theory},
Progress in Mathematics, 349,Birkh\"{a}user, 2023.
%%%%%%%%%
\bibitem{F}
T. Fack, 
\textit{Sur la notion de valeur caract\'{e}ristique}, 
J. Operator Theory, 7(1982), 307--333.
%%%%%%%%%
\bibitem{F-K}
T. Fack and H. Kosaki, 
\textit{Generalized $s$-numbers of $\tau$-measurable operators}, 
Pacific J. Math., 123(1986), 269--300.
%%%%%%%%%
\bibitem{fan}
K. Fan, 
\textit{Maximum properties and inequalities for the eigenvalues of completely continuous operators}, 
Proc. Natl. Acad. Sci. USA. 37(1951), 760-766. 
%%%%%%%%%
\bibitem{H-K}
T. Harada and H. Kosaki, 
\textit{Trace Jensen inequality and related weak majorization in semi-finite von Neumann algebras}, 
J. Operator Theory 63(2010), 129--150.
%%%%%%%%%
\bibitem{H}
F. Hiai, 
\textit{Matrix Analysis: Matrix Monotone Functions, Matrix Means, and Majorization}, 
Interdisciplinary Information Sci., 16(2010), 139--248.
%%%%%%%%%
\bibitem{H-N}
F. Hiai and Y. Nakamura, 
\textit{Extensions of nonlinear completely positive maps}, 
J. Math. Soc. Japan, 39 (1987), 367--384.
%%%%%%%%%
\bibitem{H-N2}
F. Hiai and Y. Nakamura, 
\textit{Majorizations for generalized s-numbers in semifinite von Neumann algebras}, 
Math. Z, 195 (1987), 17--27.
%%%%%%%%%
\bibitem{hiaipetz}
F. Hiai and D. Petz,
\textit{Introduction to matrix analysis and applications}, Universitext,
Springer, 2014.
%%%%%%%%%%%%%%%
%%%%%%%%%%%%
\bibitem{H-O}
A. Honda and Y. Okazaki,
\textit{Theory of inclusion-exclusion integral},
Information Sciences 376 (2017), 136--147.
%%%%%%%%%
\bibitem{kamei}
E. Kamei, 
\textit{Majorization in finite factors},
Math. Japon. 28 (1983), 495--499.
%%%%%%%%%%%%%%%%%%%%%
\bibitem{kuboando}
F. Kubo and T. Ando,
\textit{Means of positive linear operators},
Math Ann. 246(1979/80), 205--224.
%%%%%%%%%
%%%%%%%%%%%%%%%%%%%%%
\bibitem{K-S}
N. Kalton and  F. Sukochev, 
\textit{Symmetric norms and spaces of operators},
J. Reine Angew. Math. 621(2008), 81--121.  
%%%%%%%%%
%%%%%%%%%%%%%%%%%%%%%
\bibitem{L-S-Z}
S. Lord, F. Sukochev, and D. Zanin, 
\textit{Singular Traces  (first edition)},
De Gruyter Studies in Mathematics, De Gruyter, 2012. 
%%%%%%%%%
\bibitem{M-M}
F. Maccheroni and M. Marinacci,
\text{A strong law of large numbers for capacities}, 
Ann. Probab. 33(2005), 1171-1178.

%\bibitem{L}
%E.H. Lieb. 
%\textit{Inequalities for some operator and matrix functions},
%Advances in Math. . 20(1976), 174--178.
%%%%%%%%%
\bibitem{murrayvonneumann}
F. J. Murray and J. von Neumann, 
\textit{On rings of operators},
Ann Math. 37(1936), 116-229.
%%%%%%%%%
\bibitem{nagisawatatani}
M. Nagisa and Y. Watatani, 
\textit{Non-linear monotone positive maps}, 
J. Operator Theory, 87(2022), 203-228. 
%%%%%%%%%%%%%%%%
\bibitem{nagisawatatani2}
M. Nagisa and Y. Watatani, 
\textit{Non-linear traces on matrix algebras, majorization, unitarily invariant norms and 2-positivity}, 
Anal. Math. 48(2022), no. 4, 1105-1126.
%%%%%%%%%%%%%%%%
\bibitem{nagisawatatani3}
M. Nagisa and Y. Watatani, 
\textit{Non-linear traces on the algebras of compact operators and majorization}, 
Indag. Math. (N.S.) 34(2023), 724-751.
%%%%%%%%%%%%%%%%
\bibitem{nagisawatatani4}
M. Nagisa and Y. Watatani, 
\textit{Non-linear traces on semi-finite factors and generalized singular values}, 
Positivity 28(2024), no. 4, paper number 57, 40pp.
%%%%%%%%%%%%%%%%%%%%%%%%%%%%
\bibitem{nelson}
E. Nelson, 
\text{Notes on non-commutative integration},
J. Funct. Anal. 15(1974), 103-116.
%%%%%%%%%%%%%%%%%%
\bibitem{ovchinnikov}
V. I. Ovchinnikov,  
\textit{The $s$-numbers of measurable operators}, 
Funkt. Anal. Appl., 4(1970), 236-242.
%%%%%%%%%%%%%%%%%%%%
\bibitem{peng}
S. Peng,
\textit{Non-linear Expectations and Stochastic under Uncertainty}, 
Probab. Theory Stoch. Model., 95, Springer, 2019. 
%%%%%%%%%%%%%%%%%%%%%%%
\bibitem{petz}
D. Petz, 
\textit{Spectral scale of self-adjoint operators and trace inequalities}, 
J. Math. Anal. Appl. 109(1985), 74-82.  
%%%%%%%%%%%%%%%%%%%%%%%
\bibitem{powers}
R. T. Powers,  
\textit{An index theory for semigroups of $^*$-endomorphisms of $B(H)$ and type $II_1$ factors}, 
Can. J. Math. 40(1988), 86-114.  
%%%%%%%%%%%%%%%%
\bibitem{schmeidler}
D. Schmeidler, 
\textit{Integral representation without additivity}, 
Proc. Amer. Math. Soc., 97 (1986), 255-261.
%%%%%%%%%%%%%%%%%%%
\bibitem{segal}
I. Segal, 
\text{A non-commutative extension of abstract integration}, 
Ann. Math. 57(1953), 401-457. 
%%%%%%%%%%%%%%%%%%%
\bibitem{simon}
B. Simon,
\textit{Trace ideals and Their Applications},
Mathematical Surveys and Monographs Vo. 120, Amer. Math. Soc, 2005.
%%%%%%%%%%%%%%%%%%%
%\bibitem{Si}
%B. Simon,
%\textit{Loewner's Theorem on Monotone Matrix Functions},
%Springer, 2019.
%%%%%%%%%%%%%%%%%%
\bibitem{sonis}
M. Sonis, 
\textit{A certain class of operators in von Neumann algebras with Segel's measure on the projectors},
Mat. Sb.(N.S.) 84(126),(1971), 353-368.
%%%%%%%%%%%%%%%%%%%%%%%%%%%
\bibitem{Su}
M. Sugeno,
\textit{Fuzzy measures and Fuzzy integral},
a survey in M. Gupta, G. Saridis and B. Gaines (eds), Fuzzy automata  and 
decision processes, North Holland, Amsterdam, 1977,  89-102.
%%%%%%%%%
\bibitem{terp}
M. Terp, 
\textit{$L^p$-spaces associated with von Neumann algebras}, 
Notes, Copenhagen Univ. 1981. 
%%%%%%%%%%%%%%%%%%
\bibitem{teran}
P. Teran, 
\textit{Law of Large numbers without additivity}, 
Trans. Amer. Math. soc. 366(2014), 5431-5451.
%%%%%%%%%%%%%%%%%%%%%%%%%%%%%%%%%%%%%
%\bibitem{To}
%R. C. Thompson, 
%\textit{Convex and Concave functions of singular values of matrix sums}, Pac. J. 
%Math. 66 (1976), 285-290. 
%%%%%%%%%%%
%\bibitem{takesaki}
%M. Takesaki,
%\textit{Theory of operator algebras. I.},
%Encyclopaedia of Mathematical Sciences, 124. Operator Algebras and Non-commutative %Geometry, 5. Springer-Verlag, Berlin, 2002.
%%%%%%%%%%%%%%%%%
\bibitem{wang}
Z. Wang and G. J. Klir, 
\textit{Generalized Measure Theory}, Springer Science $+$ Business Media, LLC 2009. 
%%%%%%%%%%%%%%%%%
\bibitem{widder}
D. V. Widder, 
\textit{Laplace Transform}, 
Princeton University Press, 1941.
%%%%%%%%%%%%%%%%%%%%%%%
\bibitem{yeason}
F. J. Yeadon, 
\textit{Non-commutative $L^p$-spaces}, Math. 
Proc. Cambridge Philos. Soc. 77 (1975), 91-102.
%%%%%%%%%%%%%%%%%
\end{thebibliography}
\end{document}